%% file: main.tex
\documentclass{amsart}
\usepackage{amsmath}
\usepackage{amsfonts}
\usepackage{graphics}
\usepackage{float}
\usepackage{epsfig} 
\usepackage{amssymb}
\usepackage{amscd}
\usepackage[all]{xy}
\usepackage{latexsym}
\usepackage{graphicx} 
\usepackage{setspace}
\usepackage{bbm}
\usepackage{mathtools}
\usepackage{mathrsfs}
\usepackage[pdftex,
            pdfauthor={Andrew A. Port},
            pdftitle={An Introduction to Homological Mirror Symmetry and the Case of Elliptic Curves},
            pdfsubject={},
            pdfkeywords={Mirror Symmetry, Homological Mirror Symmetry, Fukaya Category, Elliptic Curve, Symplectic 2-torus},
            pdfproducer={},
            pdfcreator={}]{hyperref}

\input{LatexDefinitions}

\begin{document}
\allowdisplaybreaks
\large


\title[An Intro to HMS and the Case of Elliptic Curves]{An Introduction to Homological Mirror Symmetry\\
and the Case of Elliptic Curves}
\author{Andrew Port}
\address{Department of Mathematics, University of California, Davis, CA 95616--8633, U.S.A.}
\email{AndyAPort@gmail.com}  

\date{September 16, 2012}

\begin{abstract}
\input{Abstract}
\end{abstract}

\maketitle

\tableofcontents
\setcounter{part}{-1}
\setcounter{section}{-1}

\hbadness = 10000

\input{Intro}
\begin{ack}\input{Acknowledgments}\end{ack}
\input{Background} 
\input{BundleBackground}
\newpage
\input{Elliptic}
\newpage
\appendix
\input{HomAlg}
\newpage
\input{SpecialLags}

\newpage

\bibliographystyle{hplain}
\bibliography{main}

\end{document}

%% file: LatexDefinitions.tex
\newtheorem{thm}{Theorem}[section]
\newtheorem{cor}[thm]{Corollary}
\newtheorem{lem}[thm]{Lemma}
\newtheorem{prop}[thm]{Proposition}

\theoremstyle{definition}
\newtheorem{defn}[thm]{Definition}
\newtheorem{rem}[thm]{Remark}
\newtheorem*{ack}{Acknowledgement}

\numberwithin{equation}{section}
\numberwithin{figure}{section}

\def\Hom{{\mathrm{\rm{Hom}}}}
\def\End{{\mathrm{\rm{End}}}}

\def\rchi{{\hbox{\raise1.5pt\hbox{$\chi$}}}}

\def\vol{{\mathrm{\rm{vol}}}}

\newcommand{\cat}[1]{{\mathcal{#1}}}

\newcommand{\Nu}{{\mathcal{V}}}
\newcommand{\mF}{{\mathcal{F}}}
\newcommand{\cK}{{\mathcal{K}}}
\newcommand{\cC}{{\mathcal{C}}}
\newcommand{\cA}{{\mathcal{A}}}

\newcommand{\cDA}{{\mathcal{DA}}}

\newcommand{\cO}{{\mathcal{O}}}
\newcommand{\cM}{{\mathcal{M}}}

\newcommand{\mL}{{\mathcal{L}}}

\newcommand{\cE}{{\mathcal{E}}}

\newcommand{\mG}{{\mathcal{G}}}

\newcommand{\sL}{{\mathscr{L}}}

\newcommand{\bid}{{\mathbbm{1}}}

\newcommand{\bP}{{\mathbb{P}}}
\newcommand{\bC}{{\mathbb{C}}}
\newcommand{\bQ}{{\mathbb{Q}}}
\newcommand{\bR}{{\mathbb{R}}}
\newcommand{\bZ}{{\mathbb{Z}}}

\newcommand{\vphi}{{\varphi}}

\newcommand{\pd}{{\partial}}

\newcommand{\ddt}{{\frac{d}{dt}}}

\newcommand{\ddz}{{\frac{d}{dz}}}

\newcommand{\im}{{\mathrm{Im}\,}}

\newcommand{\zeb}{{\overline{\zeta}}}
\newcommand{\ze}{{\zeta}}

\newcommand{\spanc}[1]{\mathrm{span}_\bC\{#1\}}

\newcommand{\ppi}{{\xrightarrow{\pi}}}
\newcommand{\ptr}{{\mathcal{P}}}
\newcommand{\hol}{{Hol}}

\newcommand{\pic}{{\mathrm{Pic}}}

\newcommand{\Gl}{{GL}}

\newcommand{\dfr}[2]{{\displaystyle\frac{#1}{#2}}}

\newcommand{\inv}{{^{-1}}}

\newcommand{\pws}[1]{\begin{cases}#1\end{cases}}

\newcommand{\evat}[1]{{\left.#1\right|}}
\newcommand{\rest}[1]{{\left.#1\right|}}

\newcommand{\stb}{{\: | \:}}

\newcommand{\an}{{\: \text{and} \:}}

\newcommand{\rimp}{{\:\Rightarrow\:}}

\newcommand{\rlimp}{{\:\iff\:}}
\newcommand{\arriso}{{\xrightarrow{\sim}}}

\newcommand{\ex}{{\noindent\emph{Example:  }}}

\newcommand{\ith}{{^\text{th}}}

\newcommand{\fttext}[1]{{\footnotetext{#1}}}
\newcommand{\ftmark}{{\footnotemark}}

\newcommand{\kah}{{K\"{a}hler}}
\newcommand{\reo}{{\mathrm{Re}(\Omega)}}

\newcommand{\imm}{{\looparrowright}}

\newcommand{\ob}{{\mathfrak{Ob}}}
\newcommand{\fm}{{\mathfrak{m}}}

\newcommand{\Fuk}{{Fuk}}
\newcommand{\fuk}{{Fuk}}

\newcommand{\coh}{{\mathrm{Coh}}}
\newcommand{\db}{{\mathcal{D}^b}}
\newcommand{\dbcoh}{{\db\coh}}

\newcommand{\Ainf}{{$A_\infty$}}

\newcommand{\ext}{{\mathrm{Ext}}}

\newcommand{\thu}{{\theta_\tau}}
\newcommand{\thux}{{\theta^{(x)}_\tau}}

\newcommand{\Et}{{E_\tau}}
\newcommand{\Er}{{E^\rho}}
\newcommand{\Ert}{{E^\tau_\rho}}

\newcommand{\Ets}{E^\tau}
\newcommand{\Eq}{E_q}
\newcommand{\Lp}[1]{{\mL_{\varphi_0}^{#1}}}

%% file: Abstract.tex
Here we carefully construct an equivalence between the derived category of coherent sheaves on an elliptic curve and a version of the Fukaya category on its mirror.  This is the most accessible case of homological mirror symmetry.  We also provide introductory background on the general Calabi-Yau case of The Homological Mirror Symmetry Conjecture.

%% file: Intro.tex
\section{Introduction}

The concept of mirror symmetry\ftmark{}, in its earliest forms, was conceived in the late 1980s by physicists who observed that topologically distinct manifolds could give rise to equivalent quantum field theories and thus equivalent notions of physics.  
\fttext{The term mirror symmetry was of course used to describe other concepts in mathematics far earlier.  The concept of mirror symmetry studied in this thesis got its name from a symmetry observed in hodge diamonds of certain pairs of manifolds.  Since its naming the field has evolved quite drastically, but the name remains.}  
In 1994, roughly fifteen years after these first observations\footnote{There are many historic accounts of field's evolutions over these fifteen years (e.g. chapter 7  of \cite{yau2010shape}).}, Kontsevich proposed a mathematically rigorous framework for this symmetry based on what is now known as The Homological Mirror Symmetry Conjecture \cite{Ko94}.  The principle conjecture made in \cite{Ko94} can be thought of as a definition of what it means to be the ``mirror dual" to a Calabi-Yau manifold.  This definition (slightly weakened from its original form) can be stated as follows.\\

\emph{Let $X$ be a Calabi-Yau manifold.  A complex algebraic manifold, $\tilde{X}$, is said to be mirror dual to $X$ if the bounded\ftmark{} derived category of coherent sheaves on $\tilde{X}$ is equivalent to the bounded derived category constructed from the Fukaya category of $X$ ``(or a suitable enlargement of it)".}\\
\fttext{We will generally drop the term bounded from ``bounded derived category" as is often done in the context of mirror symmetry.}

The (derived) Fukaya category of $X$ is built from only the \kah{} structure (coupled with a B-field) of $X$, whereas the (derived) category of coherent sheaves on $\tilde X$ depends on only the complex structure of $\tilde X$.  In particular, homological mirror symmetry (HMS) is a relationship between symplectic and algebraic geometry.\footnote{We will throughout this thesis often refer to these two (symplectic and algebro-geometric) sides as the A-side and B-side respectively.  This terminology comes from a string-theoretic interpretation of mirror symmetry being a duality between type IIA and IIB string theory.}  Currently it is ``widely believed" (\cite{FOOO2009} section 1.4) that to each Calabi-Yau manifold, such a mirror dual exists and is also Calabi-Yau.\\

In the case of Calabi-Yau manifolds, as it is stated above, this equivalence (or a close version of it) has been proven completely in the cases of elliptic curves \cite{PZ} and (for $X$ a) quartic surface \cite{SeQuartic}.  Aspects of the conjecture have been proven in other cases, particularly in the cases of abelian varieties and Lagrangian torus fibrations\cite{FuAbel,KoSoTorusFibrations}.  There are also conjectured extensions of HMS to non-Calabi-Yau manifolds (e.g. \cite{ArrAntican, KKOY}), where the mirror dual of a symplectic manifold is taken to be a pair\footnote{The pair $(\tilde X,W)$ is called a \emph{Landau-Ginzberg model} and the function $W$ is called a \emph{Landau-Ginzberg superpotential}.} $(\tilde X, W)$ of a complex variety $\tilde X$ and a holomorphic function $W$.  Complete proofs of this more general HMS exist for all surfaces of genus $g\geq 2$ \cite{SeGenus2, Ef})\footnote{It is possible a complete proof of the genus zero case also exists or is implicit in related works (see \cite{Ballard} for an overview of this case).}\footnote{The author apologizes any complete proofs he is unaware of and also for the many partial results not included in the above references for both the Calabi-Yau and non-Calabi-Yau cases.}\\

%
In this thesis we will concentrate on the Calabi-Yau case of HMS.  We will give an introduction to the subject through the example where $X$ is an elliptic curve.\ftmark{}  This case is particularly accessible thanks to the fact that elliptic curves are particularly well-understood geometric objects and the intention of this document will be to provide enough detail (presented at an introductory enough level) for a geometrically-oriented mathematician to attain a good understanding of this case and its more general underlying conjecture (without the substantial and diverse prerequisites typically associated with HMS).  We will pay particular attention to explanations complementary to those included in \cite{PZ,Kreuss} and we will make effort to give our explanation in terms of the general Calabi-Yau case whenever it is possible to do so without straying too far from our goal of understanding the case of elliptic curves.\\
\fttext{Elliptic curves (i.e. one-dimensional complex tori) are the only one-dimensional Calabi-Yau manifolds.}

In the ICM talk \cite{Ko94} where Kontsevich gave his original HMS conjecture, he also gave a rough comparison of the two sides of this categorical equivalence in the case of elliptic curves.  About five years later Polishchuk and Zaslow gave a proof of this case in a weakened context, which we will describe in this thesis.  To summarize quickly, let $\tau = b+iA\in \bR\oplus i\bR_{>0}$, and let $\Ets$ denote the smooth 2-torus equipped with the symplectic form $Adx\wedge dy$ and another 2-form $bdx\wedge dy$ called the ``B-field".  The mirror dual of $\Ets$ is the complex 1-torus $\Et: = \bC/\Lambda$ given by the natural action of lattice $\Lambda = <1,\tau>$.  This ``weakened context" mentioned above is an equivalence between zeroth cohomology of the Fukaya category, $H^0(\fuk(\Et))$, (with biproducts formally added) and the derived category of coherent sheaves, $\dbcoh(\Et)$.\footnote{To be specific, this will be an equivalence between additive categories.  Generally speaking, the Fukaya category is not a true category but instead an \Ainf{}-category (or worse a curved \Ainf{}-category) only associative at the level of cohomologies.  In this one-dimensional case, however, the obstruction to its associativity vanishes and, after formally adding biproducts, we can construct an equivalence with the derived category of coherent sheaves without considering the enlargement to the derived category constructed from the Fukaya category.}  
The functor between these categories will be constructed from a surprisingly strange relation sending closed geodesics of $\Ets$ equipped with flat complex vector bundles (and some additional structure) to  holomorphic vector bundles and torsion sheaves.  In fact the indecomposable torsion sheaves will be the images of the vertical geodesics (i.e. projections of vertical lines through the quotient $\bR^2/\bZ^2$).  The image of morphisms will be described in terms of a relation sending intersection points (of geodesics) to theta functions!  Perhaps most importantly, as the composition of morphisms in the Fukaya category is given by formal sums of intersection points with coefficients determined from counting pseudo-holomorphic disks (marked at the intersection points), the functor will encode this enumerative information in (compositions of morphisms of certain coherent sheaves, which can be described in terms of) linear combinations of theta functions (with constant matrix coefficients).\\

This thesis is organized as follow.  In section 1 we give a brief review of the necessary background/definitions from symplectic geometry and complex geometry, then in section 2 we continue to a brief discussion of the necessary background from the theory of complex/holomorphic vector bundles.  Sections 3 and 4 given detailed constructions all necessary geometric structures and then all holomorphic line bundles on elliptic curves.  Section 5 and 6 discuss the Fukaya category and together give a detailed exposition of \cite{PZ}'s ``simplest example" of the simplest example of HMS.  In sections 7 and 8 we then discuss the derived category of coherent sheaves on an elliptic curve and construct the above mentioned functor between this category and an enlargement of $H^0(\fuk(\Ets))$.  We also include two appendices, the first giving a minimalist review of general derived categories and the second giving an introduction to the topic of special Lagrangian manifolds and related geometries.\\

%% file: Acknowledgments.tex
I'd like to thank my thesis committee members, Jerry Kaminker, Motohico Mulase, and Andrew Waldron.  Also I'd like to thank the many professors and other graduate students who've spent time teaching me mathematics.  In particular Andrew Waldron and David Cherney, who spent countless hours attempting to teach me quantum field theory, and Michael Penkava, who first introduced me to the homological viewpoint of mirror symmetry.  Lastly I'd like to again thank thank my advisor, Motohico Mulase, who's put a great deal of effort into developing me as a mathematician.

%% file: Background.tex
\section{Complex, Symplectic and \kah{} manifolds}

Here we will give some very basic review of complex, symplectic and \kah{} manifolds.\\

\bigskip\bigskip
\subsection{Complex manifolds}\label{sec:complex}\hfill\\

\indent A \emph{complex manifold} is a smooth, $2n$-dimensional manifold, $M$, with an atlas charting $M$ into $\bC^{n}$ such that the transition functions are holomorphic.  Equivalently, $M$ is a smooth manifold which admits a complex structure (defined below).\\

\emph{Example:  } $\bC^n$ is a complex manifold with one chart atlas given by the identity map.\\

\emph{Example:  } $\bC\bP^n$ is a complex manifold with an atlas given by the maps\\ $\phi_i:[t_0,...,t_{i-1},1,t_{i+1},...,t_n]\to(t_1,...,t_{i-1},t_{i+1},...t_n)$.

\indent An important distinction between complex and smooth manifolds is that not all complex manifolds are submanifolds of $\bC^m$.  In fact, by Liouville's theorem, any compact connected submanifold of $\bC^m$ is a point.  

\bigskip\bigskip
\subsection{Almost Complex Structures}\hfill\\

We can ``complexify" any even dimensional vector space $V$ by choosing a linear map $J:V\to V$ such that $J^2=-1$.  This gives $V$ the structure of a complex vector space with multiplication defined by $(a+ib)\cdot v = (a+Jb)v$. 

A smooth manifold $M$ is said to be \emph{almost complex} if it admits a smooth tangent bundle isomorphism $J:TM\to TM$ (called an \emph{almost complex structure}) that squares to the negative identity.  All complex (and symplectic) manifolds admit an almost complex structure, but the opposite is not true (for complex or symplectic).\\

More specifically, on an almost complex manifold, $(M,J)$, we can always find complex coordinates $z^\mu = x^\mu + iy^\mu:M\to \bC$ such that at a given point, $p\in M$, we have
$$\begin{array}{lcr}
J\dfr{\partial}{\partial x_\mu} = i\dfr{\partial}{\partial y_\mu}& \an &
J\dfr{\partial}{\partial y_\mu} = -\dfr{\partial}{\partial x_\mu}
\end{array}$$

However, it is in general not possible to find coordinates such that this is true in an entire neighborhood of $p$.  If we can find such coordinates in a neighborhood of every point $p$, then these glue together as a holomorphic atlas for $M$, and $J$ is called a \emph{complex structure}.\\

\begin{thm}[The Newlander-Niremberger Theorem \cite{NN}]
An {almost complex manifold} $(M,J)$ is a complex manifold with complex structure $J$ if and only if the Nijenhuis tensor vanishes.  I.e.
$$N_J(v,w): = [v,w] + J[v,Jw]+J[Jv,w]-[Jv,Jw]\equiv 0$$
for all $v,w\in \Gamma(TM)$, where $[\cdot,\cdot]$ is the Lie bracket of vector fields over $M$.
\end{thm}

\ex All almost complex structures on a surface are integrable.\\

\begin{rem} All symplectic manifolds admit an almost complex structure.  In fact, given a symplectic manifold $(M,\omega)$, choosing either an almost complex structure, $J$, or a Riemannian metric $g$, defines a ``compatibility triple" $(\omega,J,g)$ such that $g(u,v) = \omega(u,Jv)$.
\end{rem}

\begin{defn}
Let $f:M\to M'$ be a map between to almost complex manifolds $(M,J)$ and $(M,J')$.  $f$ is said to be $(J,J')$-holomorphic if $dfJ = J'f$.\\
\end{defn}

\bigskip\bigskip
\subsection{Symplectic Geometry}\hfill\\

The study of symplectic geometry was originally motivated by the field of Hamiltonian mechanics, which sought to describe the laws of classical mechanics in terms of ``conserved quantities" (see section \ref{sec:Ham}) on a symplectic manifold representing the space of all possible values of position and momentum for some system.  The term symplectic was introduced by Weyl in 1939 (see the footnote on p. 165 of \cite{WeylSymp}) as a greek analog to the latin root of ``complex" meaning ``braded together".  Despite having such early origins, little was understood about symplectic geometry until Gromov introduced the idea of pseudo-holomorphic curves in 1985.  Since then the fields of Gromov-Witten invariants and Floer homologies have greatly expanded our ability to understand the mathematical structure of these geometries.\\

\begin{defn}
A \emph{symplectic manifold} is a pair $(M,\omega)$ of a smooth manifold, $M$, and a closed non-degenerate 2-form on $M$, $\omega$.  We call such an $\omega$ a \emph{symplectic form}.
\end{defn}

\rem{The non-degeneracy of this anti-symmetric form enforces that all symplectic manifolds are even dimensional and at each $p\in M$ there exists a basis ${e_1,...e_n,f_1,...,f_n}$ of $T_pM$ such that $\omega(e_i,f_j) = \delta_{ij}$ and $\omega(e_i,e_j) = \omega(f_i,f_j) = 0$ for all $i,j$.}

\noindent\emph{Examples:}\\
\begin{itemize}
\item $2n$-dimensional Euclidean space: $\bR^{2n}$ with its standard coordinates and\\
$\omega = \sum_{i=1}^{n}dx_i\wedge dx_{i+n}$.\\

\item Any 2-surface with any non-vanishing 2-form.\\

\item  The $2n$-torus\footnote{After putting a complex structure on it, we will refer to this manifold as a complex $n$-torus or, in the case $n=1$, an elliptic curve}:  $\bC^{n}/\bZ^{2n} \overbrace{\cong S^1\times ...\times S^1}^\text{$2n$-times}$ with its standard angular coordinates and $\sum_{i=1}^n\omega = d\theta_i\wedge d\theta_{i+n}$\\

\item The cotangent bundle $T^*M$ of any smooth $n$-dimensional manifold $M$ can be given a symplectic structure.  Letting $(x_1,...x_n)$ denote local coordinates on $M$, we can define local coordinates on $T^*M$ by the map\\
$(x,\xi)\mapsto (x_1,...,x_n,\xi_1,...,\xi_n)$, where $\xi_i$ are the components of $\xi$ with respect to the basis local coordinate from, i.e. $\xi = \sum\xi_i dx_i$.\\

$\omega = \sum dx_i\wedge d\xi_i$ then definines a symplectic form on $T^*M$.
\end{itemize}
\hfill\\

\noindent\emph{Non-Examples:}\\
\begin{itemize}
\item Any odd dimensional or non-orientable manifold.\\

\item \emph{By Stoke's theorem:} Any compact connected manifold with vanishing second cohomology (e.g. $S^n$ for $n>2$).  A symplectic form would then have to be exact, and thus, letting $\eta$ denote its primitive, give rise to the exact volume form $\omega^n = d(\eta\wedge \omega^{n-1})$.\\
\end{itemize}

\begin{thm}[Darboux's theorem]  
Let $(M,\omega)$ be a $2n$-dimensional symplectic manifold and $p\in M$.  Then there exist local coordinate chart $(U,(x_1,...,x_n,y_1,...,y_n))$ such that $p\in U$ and, on $U$,
$$\omega = \sum_{i=1}^n dx_i\wedge dy_i$$
\end{thm}\begin{proof}A proof of this theorem can be found in most introductory texts on symplectic geometry (e.g. \cite{daSilva}).
\end{proof}
\hfill\\
\begin{rem}  This says that all symplectic manifolds locally all look alike.  For purpose of comparison, any Riemannian metric can be made (by choice of coordinates) to look like the Euclidean at a single point, but not necessarily in an entire neighborhood.  If we can find such a local coordinate patch around every point, then our manifold is, by definition, \emph{flat}.\end{rem}    

\bigskip\bigskip
\subsection{Lagrangian submanifolds and Hamiltonian flows}\label{sec:Ham}\hfill\\

The objects of the Fukaya category are closed Lagrangian submanifolds (with some additional structure we will describe later).  The submanifolds are usually either required to be ``special" (in the sense described in appendix \ref{sec:SL}) or else only considered up to ``Hamiltonian equivalence" in the sense we will describe now.\\
 
\begin{defn}  Let $(M,\omega)$ be a symplectic manifold.  A submanifold $L\subset M$ is called \emph{Lagrangian} if $\rest{\omega}_L\equiv 0$.
\end{defn}

Let $V$ be any vector field on a smooth manifold $M$.  We can locally (in time) solve for the ``flow" of a point $x\in M$ in the direction of $V$ (i.e. solve $\dot \gamma(t) = V_{\gamma(t)}$, $\gamma(0) = x$).  It is simple to show that this gives us a family of diffeomorphisms $\psi_t:M\to M$ called the \emph{flow} of $V$.  We call $V$ \emph{complete} if its flow exists for all time.  A vector field, $V$, on a symplectic manifold, $(M,\omega)$, is called a \emph{symplectic} if its flow preserves $\omega$, i.e. $\psi_t^*\omega = \omega$.\\
By Cartan's formula for the Lie derivative we have 
$$\ddt \psi_t^*\omega = \mathrm{Lie}_V \omega = d(V\llcorner \omega) +V\llcorner d\omega = d(V\llcorner \omega)$$
Where we use $\llcorner$ to denote the interior product, i.e. $V\llcorner \omega:W\mapsto \omega(V,W)$.\\

This tells us that $V$ is symplectic if and only if $V\llcorner \omega$ is closed.  If $V\llcorner \omega$ is exact, then there exists some function $H:M\to \bR$ such that $dH = V\llcorner \omega$.  Moreover, this function $H$ is preserved by the flow, i.e. $\ddt \psi_t^* H = dH(V) = 0$.  In this case, $\psi_t$ is called a \emph{Hamiltonian flow}, $H$ a \emph{Hamiltonian function}, and $V$, a \emph{Hamiltonian vector field}.\\

Notice that $\omega$ defines a perfect pairing between vectors and covectors.  Thus given any function $H$, we can find a vector field $V$ such that $dH = V\llcorner \omega$.\\

\begin{defn}
Two Lagrangian submanifolds, $L$ and $L'$ are said to be \emph{Hamiltonian equivalent} if there exists a Hamiltonian flow $\psi:I\times M\to M$ such that $\psi_t(L) = L'$ for some $t\in I$.\\ 
\end{defn}

\bigskip\bigskip
\subsection{\kah{} Geometry}\hfill\\

A smoothly varying positive-definite hermitian inner product on the tangent bundle of an almost complex manifold is called a \emph{Hermitian metric}.  A (resp. almost) complex manifold equipped with such a structure is called a \emph{(almost) Hermitian manifold}.\\

An almost Hermitian manifold $(M,J,h)$ has an associated Riemmanian metric and non-degenerate 2-form given by its symmetric and anti-symmetric parts, i.e. $g = \frac{1}{2}(h+\bar{h})$ and $\omega = \frac{i}{2}(h-\bar{h})$.  Morover, as any smooth manifold admits a Riemannian metric $g$, the existence of an almost complex structure $J$ gives us the existence of a non-degenerate 2-form $\omega(X,Y) = g(JX,Y)$ and thus a Hermitian metric $h = g-i\omega$.\\

\begin{defn}
A \emph{\kah{}} manifold is an almost Hermitian manifold that satisfies an integrability condition that can be stated in the following three equivalent ways:
\begin{itemize}
\item $\omega$ is closed and $J$ is integrable
\item $\nabla J = 0$
\item $\nabla \omega = 0$
\end{itemize}
where $\nabla$ is the Levi-Cevita connection of $g$.\\
\end{defn}

\ex All Riemann surfaces are \kah{}.\\

%
%



\bigskip\bigskip
\subsection{Calabi-Yau Manifolds}\label{sec:CY}\hfill\\

Studied by mathematicians since at least the 1950s \cite{Calabi54,Calabi57}, Calabi-Yau manifolds have, since the 1980s \cite{CHS85}, been of particular importance to physicists studying the subject of superstring theory.  Superstring theory is a unified theory seeking to describe all elementary particles and fundamental forces of nature by modeling particles as vibrating strings.  For the superstring model to be a consistent physical theory (e.g. to predict massless photons) it is necessary for space-time to be ten-dimensional.  Experimental observation has led us to believe that our space-time locally looks like four-dimensional Minkowski space, $M^4$, (i.e. $\bR^4$ equipped with the pseudo-Riemannian metric $d\vec{x}^2-dt^2$).  Thus if these extra six dimensions do exist, it is expected that space-time locally look like the product $M^4\times X$ for $X$, speaking informally, some compact space with dimensions of tiny length (on the order of the Planck length).  Different choices of $X$ (and, of course, product structure) lead to different ``effective" four-dimensional theories.  It was shown in \cite{CHS85}, that (under certain assumptions about the product structure of space-time) $X$ must be taken to be a Calabi-Yau manifold in order for this effictive four-dimensional theory to admit an $N=1$ supersymmetry (a property popularly desired by physicists).\\

We define an $n$-dimensional \emph{Calabi-Yau} manifold to be a compact $n$-dimensional \kah{} manifold admitting a nowhere vanishing holomorphic $n$-form, which we call its \emph{Calabi-Yau form}.\\

This definition of Calabi-Yau implies that our \kah{} structure has a vanishing first Chern class and thus can be thought of as Ricci-flat by ``The Calabi Conjecture", which we state here as proven by Yau \cite{Yau77,Yau78}.\\
\begin{thm}\label{thm:CY}Let $M$ be compact complex manifold admitting \kah{} form $\omega$ and suppose $\rho'$ is some real, closed $(1,1)$-form on $M$ such that $[\rho']=2\pi c_1(M)$.  Then there is a unique \kah{} form $\omega'$ on $M$  such that $[\omega]=[\omega']\in H^2(M,\bR)$ and $\rho'$ is the Ricci form of $\omega'$.\\
In particular, if the first Chern class vanishes on $M$ (w.r.t. a \kah{} metric), then $M$ admits a unique cohomologically equivalent Ricci-flat (i.e. $\rho=0$) \kah{} structure.\end{thm}

\ex In one-dimension this implies that that our compact Riemann surface is parallelizable.  Thus the only examples are those of genus one.\\



%% file: BundleBackground.tex
\section{Complex and Holomorphic Vector Bundles}\label{sec:VB}

Studying HMS requires a good understanding of both complex and holomorphic vector bundles.  We will give some back review of these subjects here as relevant to this thesis.\\
\bigskip\bigskip
\subsection{Complex Vector Bundles}\hfill\\
\begin{defn}{A rank $k$ \emph{complex vector bundle} is a vector bundle $\bC^k\to E\xrightarrow{\pi} M$ whose transition functions are $\bC$-linear.}\end{defn}

\ex  \emph{The trivial bundle}, $M\times\bC^k\xrightarrow{\text{proj}_1} M$, is a rank $k$ complex vector bundle.  The complexified tangent bundle of an almost complex manifold is also a complex vector bundle.  These two examples are rarely the same.  In fact, the complex (one) dimensional torus is the only Riemann surface where this is true (for $k=1$).

Let us classify all complex vector bundles on the $k$-sphere.\\
\begin{lem}{Every vector bundle over a contractible base is trivial.}\end{lem}\begin{proof}{If this were not true, our definition of vector bundle would not make sense.}\end{proof}
\begin{thm}{There is a 1-1 equivalence between homotopy classes of functions $f:S^{k-1}\to \Gl_k(\bC)$ and complex vector bundles on $S^k$.}\end{thm}
\begin{proof}{Let $E\to S^k$ be any rank $n$ complex vector bundle.  By the above lemma, $E$ can be covered by two trivializations, $\psi_\pm:\pi\inv(D^k_\pm)\to D^k_\pm\times \bC^n$,  where $D_+\cap D_- \cong (-\epsilon,\epsilon)\times S^{k-1}$.  So the transition functions of this two chart atlas are defined by a homotopy $f_t:S^{k-1}\to \Gl_k(\bC)$ (for $t\in (-\epsilon,\epsilon)$).\footnote{Given a complete set of transition functions, we can, of course, reconstruct the bundle they came from.  This homotopy, $f_t$, is sometimes called a \emph{clutching function}.}  All that is left to do then is show this homotopy class is well defined, i.e. if we take any two chart atlas, it gives us the same homotopy class.  To see this notice that we can contract $\phi_\pm$ to a constant map over a point.  As $\Gl_n(\bC)$ is path connected, all such charts are homotopy equivalent.\footnote{$\Gl_n(\bR)$, on the other hand, has not one but two path connected components.  Though we similarly classify all \emph{orientable} real vector bundles on the $k$-sphere.}}\end{proof}

\begin{rem}
From this we have that all complex vector bundles on the circle are trivial (we can also see from the proof that we have two choices of real bundle structure - the cylinder or M\"{o}bius band).  The above technique can be used for any manifold that can be covered by two contractible charts.  For example, all bundles on $\bC^*$ are also trivial.  Actually, we can generalize this (as is done in the finally of \cite{BottTu}) to manifolds having a ``good" cover of $r$ charts (i.e. manifolds of finite type) to find a bijection between rank $k$ complex vector bundles and homotopy classes of maps $M\to Gr_k(\bC^n)$ for any $n\geq rk$.  In fact, we can classify complex vector bundles over an arbitrary manifold using the infinite Grassmannian (see \cite{HatcherVBK}).
\end{rem}

Let $E\to M$ be a real rank $k$ vector bundle over an $n$ dimensional manifold $M$.\\
If $k>n$, then there exists a nonvanishing global section of $E$.  Similarly if $E$ is complex and has fiber $\bC^k$, then there exists a global section if $k>n/2$.  In particular this tells us then any vector bundle looks like $E = E'\oplus I_{k-n}$ where $I$ is the trivial bundle and $E'$ is some bundle with an $n$ dimensional fiber.\\

Note that taking any section of $E$ we can use the transversality theorem to assume it is transversal to the zero section and thus, locally, has finite zeros.\\
\begin{defn}\label{def:deg}
The \emph{degree} of a complex vector bundle $E$ is measured using as section $\sigma$ which is transverse to the zero section.  $\deg(E) = \sum (-1)^{sgn(det(d\sigma_{p}))}$
\end{defn}  
This is well-defined and equivalent to the degree homomorphism given by the first Chern class (or the exponential sheaf sequence).\\


\bigskip\bigskip
\subsection{Connections on Vector Bundles}\hfill\\


\begin{defn}Let $G$ be a topological group and $M$ a manifold.  A \emph{principle $G$-bundle} over $M$ is a fiber bundle $P\ppi M$ equipped with a fiber-preserving continuous right action of $G$ on $P\ppi M$ which acts freely and transitively on its fibers.
In other words, we have a continuous right action of $G$, $\rho:P\times G\to P$, such that $\pi\circ \rho = \pi$ and at any point in $p\in P$, the induced map $\rho_p:G\to \pi^{-1}(\pi(p))$ is a homeomorphism.
\end{defn}

\ex Consider the frame bundle, $GL(E)$, of some vector bundle $F\to E\to M$.  $GL(E)$ is a principle $\Gl(F)$-bundle under the action $(p\cdot g)(v) = p(g\cdot v)$.  This is a useful way to think about vector bundles as is shown by the following theorem.\\


\begin{thm}
\label{thm:monodromyguage}
The gauge equivalence classes of flat connections on a principle $G$-bundle over a connected manifold $M$ are in one-to-one correspondence with the conjugacy classes of representations of $\pi_1(M)\to G$.
\end{thm}
\begin{proof}A simple proof of the fact (from the perspective of distributions) can be found in section 2.1 of \cite{morita2001geometry}.\end{proof}

\begin{rem}The map $\pi_1(M)\to G$ is known as the \emph{monodromy representation} of the connection.  The \emph{holonomy group} of a connection (based at point $x\in M$) is defined in terms of the parallel transport operator, $\ptr$, as $\hol = \{\ptr_\gamma\stb \gamma \text{ is a loop based at } x\}$.  The \emph{restricted holonomy group}, $\hol^0$ is the normal subgroup given by contractible loops.  The monodromy representation is then the natural surjection $\pi_1(M)\to \hol/\hol^0$.  A connection is flat if and only if $\hol^0$ is trivial; thus in our case, this surjection is also sometimes called the \emph{holonomy homomorphism}.
\end{rem}

\begin{cor}\label{thm:S1con}  The gauge equivalence classes of flat connections on rank $k$ complex vector bundles over  $S^1$ are in one-to-one correspondence with the conjugacy classes of $\Gl_k(\bC)$.\footnote{Recall that the conjugacy classes of $\Gl_k(\bC)$ are determined by the Jordan canonical form.}\\
\end{cor}\begin{proof}  $\pi_1(S^1)\cong \bZ$ and any representation $\bZ\to \Gl_k(\bC)$  is determined by the image of $\bid$.  For this same reason, two representations are conjugate if and only if the images of their generators are conjugate.\\ 
\end{proof}

\begin{rem}{\label{rmk:monodromy}The image of $\bid$ is sometimes referred to (e.g. \cite{Kreuss,PZ}) as the ``monodromy operator" or simply the ``monodromy" of the pair connection.  In section \ref{sec:Elliptic}, when discussing the Fukaya category for an elliptic curve, we will restrict our attention to just flat connections (on vector bundles over $S^1$) whose monodromy operators have only eigenvalues of unit modulus.}\end{rem}

\subsubsection{Local Systems}\hfill\\

A \emph{sheaf of locally constant functions} on a topological space, $X$, is mapping (satisfying the sheaf axioms) from the open sets of $X$ to some fixed module.  E.g. the space of constant functions $\{f:X\to \bR \stb f(x) = a \text{ for all }x\in X\}$ forms such a sheaf when we consider all restrictions of these functions to open sets in $X$.  A \emph{local system}\footnote{also often called a \emph{local coefficient system}} is the more general concept of a sheaf which locally looks like a sheaf of constant functions, but globally may be twisted.  For example, considering a flat vector bundle equipped with connection $\nabla$, the space of local sections, $s$, that satisfy $\nabla s = 0$ form a local system.
In mirror symmetry literature the terms flat complex vector bundle (equipped with connection) and \emph{local system} are often used synonymously.\\
We will often use the term ``local system" in this thesis to describe a Lagrangian submanifold equipped with a complex vector bundle and a flat connection.\\

\bigskip\bigskip
\subsection{Holomorphic Vector Bundles}\label{sec:HB}\hfill\\

\begin{defn}
A \emph{holomorphic vector bundle} is a complex line bundle over a complex manifold whose transition functions are holomorphic.
\end{defn}

Some examples are given by the canonical line bundle and the cotangent bundle of any complex manifold.\\

It is important to note that, while every complex vector bundle over a Riemann surface\footnote{This is not true over a general complex manifold.}, admits a holomorphic structure, this structure is not unique.  This does not complicate working with line bundles much as the space of line bundles over a Riemann surface, $C$, (of genus $g$) forms a group, $\pic(C)\cong \bZ\times J(C)$ where $J(C)\cong \bC^g/\bZ^{2g}$ denotes the space of topologically trivial line bundles.  The group structure of the \emph{Picard group}, $\pic(C)$, is given by the tensor product (and thus inversion by the dual space operator).  The degree function gives a homomorphism to $\bZ$ with kernel, the \emph{Jacobian variety}, $J(C)$.\\

In contrast to the complex case, holomorphic vector bundles (over Riemann surfaces) do not all split into sums of holomorphic line bundles.  This makes holomorphic vector bundles much more difficult to work with than their complex analogs.  Luckily in the genus one case these objects have a fairly simple classification discovered by Atiyah \cite{Avbe} (we will state his result in section \ref{sec:EllDb}).\\

Our ability to use Atiyah's classification in a black box manner will be largely aided by the fact that we can pull all holomorphic bundles on $\bC/\bZ^2\cong \bC^*/\bZ$ back to $\bC^*$ where all vector bundles are trivial.\footnote{We pull back to $\bC^*$ as opposed to $\bC$ as a convenience (we then only have to consider a $\bZ$ action).}\\

\begin{thm}
All holomorphic vector bundles on a non-compact Riemann surface are holomorphically trivial.
\end{thm}\begin{proof}  For a simple complex analytic proof by induction on rank (see theorem 30.4 on p. 229 of \cite{Forster}).\end{proof}
\noindent
We will also later make some minor use of the line bundle-divisor correspondence.\\
A \emph{divisor} on a Riemann surface, $C$, is a finite linear combination of points in $C$ with integer coefficients.  To divisors $D$ and $D'$ are linearly equivalent if there is a meromorphic function, $f$, on $C$ such that $D-D' = ord(f)$.\footnote{By $ord(f)$ we mean the sum of all zeros and poles weighted by their orders.}  A proof of the following fact can be found in any introduction to algebraic geometry or Riemann surface theory (e.g. \cite{VarolinRS}).\\ 
\begin{thm}[Line bundle-Divisor Correspondence]
There is a bijection between holomorphic line bundles and divisors modulo linear equivalence.\\
\end{thm}

\bigskip\bigskip
\subsection{Sections of a holomorphic vector bundles}\hfill\\

The degree of a complex vector bundle is a topological invariant, but is very useful for studying holomorphic bundles.  For example, if a holomorphic line bundle has a global holomorphic section, then it must have a positive degree (by the Cauchy-Riemann equations).  In particular, all global holomorphic sections must vanish at the same number of points, and this number, if nonzero, is equal to the degree.  If the degree is negative, then the line bundle admits no global holomorphic sections.  This can be seen by using the Cauchy-Riemann equations with definition \ref{def:deg}.\\



\begin{thm}[Grothendieck's Vanishing Theorem]
If $\Nu$ is a holomorphic vector bundle over an $n$-dimensional complex manifold, $X$, then $H^k(X,\Nu) = 0$ for $k>n$.\end{thm}\begin{proof}A proof of this can be found in section III.2 of \cite{HartAlgGeo}.\end{proof}

\begin{rem}
This means on a Riemann surface, we know $H^k(C,\Nu) = 0$ unless $k=0,1$.  Both of these spaces also have nice descriptions.  $H^0(C,\Nu) = \Gamma(\Nu)$ (by the definition of sheaf cohomology) and by the below Serre Duality, $H^1(C,\Nu) \cong H^0(C,K_C\otimes \Nu^*)$.  In the elliptic curve case, we then have that $H^1(\Nu)\cong \Gamma(\Nu^*) = \Hom(\Nu,\cO_C)$.\\
It is also worth noting that $H^1(X,\cO^*_X)$ parameterizes holomorphic line bundles, where by $\cO^*$ we mean the sheaf of non-vanishing holomorphic functions.  A similar statement can be made for real and complex line bundles.
\end{rem}

\begin{thm}[Riemann-Roch Theorem]
If $\Nu$ is a holomorphic vector bundle over a Riemann surface $C$, then $\chi(C,\Nu) = \deg(\Nu) + (1-g)rank(\Nu)$.
\end{thm}
\noindent
Here $\chi$ is the Euler characteristic of $\Nu$, defined as $\chi(\Nu): = \sum_k(-1)^kh^k(\Nu)$ for\\ $h^k(\Nu) := \dim(H^k(C,\Nu))$.\\
\begin{proof}This particular statement is sometimes called ``Weil's Riemann-Roch theorem" and has a generalization to higher dimensional compact complex manifolds called the Hirzebruch-Riemann-Roch theorem.  A direct proof of this particular statement can be found on page 65 of \cite{gunning1967lectures}.\end{proof}

\begin{thm}[Serre Duality]
If $\Nu$ is a holomorphic vector bundle over an $n$-dimensional compact complex manifold, $X$, then $H^k(X,\Nu) \cong H^{n-k}(X,K_X\otimes\Nu^*)^*$. 
\end{thm}
\noindent
Here $K_X$ denotes the canonical line bundle of $X$ (i.e. $K_X := \bigwedge^{n,0}X$).
\begin{proof}
A proof of this statement of Serre duality, sometimes called ``Kodaira-Serre duality", can be found in section 4.1 of \cite{huybrechts2004complex} or section 1.2 of \cite{GHAlGeo}.
\end{proof}

\begin{rem}
\label{rmk:degdim}
Recall that on any Calabi-Yau manifold, $K_X$ is the trivial bundle.  Combining these three theorems, we have, for any holomorphic vector bundle, $\Nu$, over a genus one compact Riemann surface, $C$, $\deg(\Nu) = h^0(\Nu) - h^1(\Nu) = h^0(\Nu) - h^0(\Nu^*)$.
In particular, if $\Nu$ is rank one, then $\deg(\Nu)= \pws{h^0(\Nu) & \text{if } \deg(\Nu) > 0\\-h^0(\Nu^*)&\text{if } \deg(\Nu)<0}$.
\end{rem}

%% file: Elliptic.tex
\section{A Review of Elliptic Curves}

We define an elliptic curve to be the Riemann surface given by the quotient $\bC/\Lambda$ and its inherited complex structure, for $\Gamma$ some $\bZ^2$ lattice.  All complex structures (up to biholomorphic equivalence) on a smooth 2-torus are given in such a way.  Choosing $\Lambda = <1,\tau>$, we get a bijective relationship between elliptic curves and choices of $\tau$ with $\im(\tau)>0$ up to modular transformation.  An explanation of these facts can be found in nearly any introductory text discussing the theory of Riemann surfaces (e.g. \cite{Miranda,Farkas}).  Keeping this bijection in mind, we will often use the notation $E_\tau:=\bC/<1,\tau>$ and will always assume, for reasons of convenience, $\im(\tau)>0$.\\

The complex, metric, and symplectic structures on $\bC$ are all preserved by translations.  Thus they all descend to $\bC/\Lambda$.  We will talk about this in more detail below, but for now let us just show we have a holomorphic atlas.\\

\bigskip\bigskip
\subsection{Complex structures on Elliptic curves}\hfill\\

Observing that the quotient map $\bC\xrightarrow{\pi} \bC/\Lambda$ is open, we are given an holomorphic atlas $\bigcup_{z_0\in \bC} \{(\pi(B_{\epsilon}(z_0)),\evat{\pi}_{\pi(B_{\epsilon} (z_0))}^{-1})\}$ where $\epsilon$ is some small number such that $\epsilon<|\lambda|$ for all $\lambda \in \Lambda$.\\
Our transition functions then differ from the identity map by an element of $\Lambda$ and thus are holomorphic.\\


\rem{Sometimes we will find it convenient to reparameterize our curve as $E_{\tau} = \bC^*/\bZ =: E_q$ using the homomorphism $u:\bC\to \bC^*$ given by $u:z\mapsto e^{2\pi i z}$ so that the quotient is over multiplication by $q=e^{2\pi i
\tau}$.  Throughout our discussion of HMS for the elliptic curve, $\tau$ and $q$ will always be assumed to have this relationship, $q=e^{2\pi i\tau}$.}

\bigskip\bigskip
\subsection{The Calabi-Yau Structure of an Elliptic curve}\hfill\\

As the standard Hermitian structure on $\bC$ is invariant under translations, it descends to a Hermitian structure on $\bC/\Lambda$.  The Hermitian metric's associated \kah{} form is automatically closed (being a top form) and compatible with our descendant complex structure.  Thus $\bC/\Lambda$ is a \kah{} manifold.\\

Let $\xi_{z_0}  = \evat{\pi}_{B_\epsilon(z_0)}$ denote our local coordinate charts used above.  Notice that $d\ze$ is well-defined globally; this tells us $\bC/\Lambda$ is a Calabi-Yau Manifold.\\
From theorem \ref{thm:CY}, we then know that given a \kah{} structure on $\Et$, there is a unique cohomologous Ricci-flat \kah{} structure.\\
As $H^2(\Et,\bR) = \bR$, we see that all Ricci-flat \kah{} structures on $\Et$ must be the positive\ftmark{} scalar multiples of the one we've pulled down from $\bC$.  Thus our \kah{} structure on complex manifold $\Et$ is determined by its volume, which we will denote by $A$.\\
\fttext{While a negative multiple would still give a symplectic form, it would not be compatible with a metric.  Better said, the sign of the symplectic volume corresponds to whether the orientation of our symplectic structure agrees with that of our Riemannian.  It will become apparent later that this convention, $A>0$, on the symplectic structure can be see to mirror the convention that $\im(\tau)>0$ on the complex structure.}

Choosing the \kah{} structure on $\Et$ given by a volume of $A$, we have: 
$$\begin{array}{rcl}
\omega &=& \text{vol}_g = Adx\wedge dy = \frac{i}{2}Ad\ze\wedge d\zeb\\
g &=& A(dx\otimes dy + dy\otimes dx) = -\frac{i}{2}A(d\ze\otimes d\ze -d\zeb\otimes d\zeb )\\
\end{array}$$

As $\Et$ is compact, it has no non-constant holomorphic functions.  Thus all CY forms are constant multiples of $d\ze$.  This can of course be generalized to any Calabi-Yau manifold.\\
To construct the Fukaya category, we will use the fact that our manifold is Calabi-Yau/\kah{} often, and sometimes it is even helpful to fix a Calabi-Yau form (e.g. see appendix \ref{sec:phaseL}) or complex structure;  the final result, however, will only depend on our choice of symplectic structure and B-field (see section \ref{sec:Elliptic} for a definition of B-fields).\\


\section{Holomorphic Line Bundles on Elliptic Curves}
\subsection{Theta Functions}\hfill\\

As $E_\tau$ is compact, any holomorphic function, $E_\tau\to \bC$, is constant.  We can construct all meromorphic functions on $E_\tau$ in a similar fashion to how one would on projective space.  In this analogy, the role of homogeneous polynomials are played by theta functions. That said, we are particularly interested in theta functions because they will give us bases for our B-side homsets.\\

Consider the classical Jacobi theta function $\thu:\bC\to \bC$ given by 
$$\thu(z):=\sum_{m\in \bZ}e^{\pi i [m^2\tau + 2mz]} = \sum_{m\in \bZ}q^{\frac{1}{2}m^2}u^m$$
As we are always assuming $\im(\tau)>0$, this converges for all $z\in \bC$.\\ 

\begin{prop}\hfill\\
\label{thm:thetaprops}
\begin{itemize}
\item[i)] $\thu(z+1) = \thu(z)$
\item[ii)] $\thu(z+\tau) = e^{-\pi i[\tau+2z]}\thu(z) = q^{-1/2}u\inv\theta_q(u)$
\item[iii)] $\thu(z)=0\:\iff\: z = \frac{1}{2}+\frac{\tau}{2} + (k+\ell\tau)$ for $k,\ell \in \bZ$ (i.e. $\theta_q(u) = 0\rlimp u = -q^{\frac{1}{2}+\ell}$) and these zeros are simple.
\item[iv)] $\thu(-z) = \thu(z)$
\item[v)] $\thu(\frac{\tau}{2} - z) = e^{2\pi i z}\thu(\frac{\tau}{2} + z)$
\end{itemize}
\end{prop}
\begin{proof}\hfill\\
(i) and (iv) are trivially true.\\
(ii) can be seen by a shift of the sum index (sending $m\to m-1$).\\
(iii) can be seen be showing the integral of $\theta'/\theta$ around a fundamental parallelogram is equal to $1/(2\pi i)$.\\
(v) can be seen by combining (ii) and (iv).
\end{proof}

Letting $\thux$ denote the translation $\thu$ with zeros at $x\,\,+<1,\tau>$\\ (i.e. $\thux(z):=\theta(z-1/2-\tau/2 - x)$, we have the following result:
Theta functions are analogous to homogenous polynomials on projective space in that all meromorphic functions on an elliptic curve can be described by their ratios.
\begin{thm}  A function $R:E_\tau\to \bC$ is meromorphic if and only if it can be written in the form
$$R(z) = \frac{\prod_{i=1}^d{\thu^{(x_i)}(z)}}{\prod_{i=1}^d{\thu^{(y_i)}(z)}}$$
for $\{x_i\}$ and $\{y_i\}$ any finite sets of $d>0$ complex numbers such that $\sum_i x_i -\sum_i y_i = \in \bZ$.
\end{thm}
\begin{proof}
Showing that any ratio of theta functions is well-defined (and thus meromorphic) on $\Et$ if and only if it meets the above criteria can be seen simply by checking when such ratios are invariant under translation by $\tau$; a straight-forward computation.\\  A simple proof that all meromorphic functions are such ratios can be found in \cite{Miranda}.
\end{proof}

We are interested in theta functions for the reason that they descend to holomorphic sections of line bundles on $\Et$ and, as we will see below, will give bases for our B-side homsets.  As we construct our mirror functor in the following sections, we will need to introduce the following slight generalization of $\theta$.
$$\theta[a,z_0](\tau,z) = \sum_{m\in \bZ}e^{\pi i [(m+a)^2\tau + 2(m+a)(z+z_0)]}$$

\noindent Notice that $$\theta[a,0](\tau,z) = e^{\pi i [a^2\tau + 2az]}\theta_\tau(z+a\tau)$$ so properties (i) and (ii) still hold.  In particular $$\theta[a,0](\tau,z+\tau) = e^{-\pi i[\tau+2z]}\theta[a,0](\tau,z).$$\\
Also later we will need to use that 
$$\theta[a+b,0](\tau,z) = e^{\pi i [a^2\tau + 2az]}\theta[b,\tau z](\tau,z+a\tau).$$

Thinking of theta functions as line bundle morphisms, the following \emph{addition formula} (II.6.4 of \cite{MumTata}) will all allow us to compute compositions on the B-side.\\
\begin{prop}\label{thm:thetaadd}  Let $a,b\in \bQ$, $n_1,n_2\in \bZ_{\geq 1}$, $k = n_1+n_2$, and $c_j = jn_1+a+b$.
$$\begin{aligned}\theta\left[a/n_1,0\right]&(n_1\tau,z_1)\cdot\theta\left[b/n_2,0\right](n_2\tau,z_2)\\
 &= \sum_{j\in \bZ/k\bZ}\theta\left[\frac{c_j}{k},0\right](k\tau,z_1+z_2)\cdot \theta\left[\frac{n_2c_j-kb}{n_1 n_2 k},0\right](n_1n_2k\tau,n_2z_1-n_1z_2)
\end{aligned}$$
\end{prop}


\noindent
For example, if $n_1=n_2=1$, $a=b=0$ and $z=z_1=z_2-x$, this gives us
\begin{equation}\label{eq:thetaprodex}
\thu(z)\cdot \thu(z+x) = \theta_{2\tau}(x)\theta_{2\tau}(2z+x) + \theta[1/2,0](2\tau,x)\theta[1/2,0](2\tau,2z+x)
\end{equation}

\bigskip\bigskip
\subsection{The Category of Holomorphic Line Bundles on Elliptic curves}\label{sec:EllLB}\hfill\\

Letting $\pic^d(C)$ denote the space of degree $d$ line bundles on a curve $C$ of genus $g$ and $\mL_d\in \pic^d(C)$, we have an isomorphism $J(C):=\pic^0(C)\arriso \pic^d(C)$ given by $\mL\mapsto \mL\otimes_{\cO_C} \mL_d$.  By the Abel-Jacobi theorem $J(C)$, usually called the \emph{Jacobian variety} of $C$, is isomorphic to a complex torus of dimension $g$.  In the case of an elliptic curve $\Et$, we then have that $\Et\cong J(\Et)$.  To see this we can use the line bundle-divisor correspondence and the above isomorphism $J(\Et)\cong \pic^1(\Et)$.  The degree one line bundles $\mL_x$ and $\mL_y$ associated with divisors $x\in \Et$ and $y\in \Et$ are equivalent if and only if $x=y$.\footnote{The statement $\mL_x \cong \mL_y\rlimp x=y$ is true on any curve $C\ncong \bP^1$.  We can think of any meromorphic function $f$ on $C$ as holomorphic map $f:C\to \bP^1$.  If $ord(f)=x-y$, then this holomorphic map is degree one and thus an isomorphism $C\arriso \bP^1$.}  The homomorphism $\mL_x:\Et\to \pic^1(\Et)$ is then 1-1.  It is also surjective (and thus prove this case of the Abel-Jacobi theorem) we need only to observe that, combining the Riemann-Roch theorem and Serre duality, $\deg(\mL)= h^0(\mL)$ and so, if $\deg(\mL)=1$, then $\mL$ has a global holomorphic section vanishing only at one point.\\
Note as we continue to discuss holomorphic line bundles, we will generally refer to them as simply line bundles and their global holomorphic sections as simply, sections.\\

By the construction of this isomorphism, fixing any degree one line bundle, $\mL$, we can parameterize all degree $n$ line bundles by $(t_x^*\mL)\otimes\mL^{n-1}$, where $t_x$ is the automorphism of $\Et$ given by translation by any particular $x\in \Et$ and $\mL^k: = \mL^{\otimes k}$ is the $k\ith$ tensor product of the line bundle with itself.\\
For our particular task of building a mirror functor we will pick a special $\mL$.  In particular, the one which holds $\theta_\tau$ as a section.\\

For convenience now we will use our reparameterization $E_q := \bC^*/u\sim qu$.\\ 
Any holomorphic function $\vphi:\bC^*\to \bC^*$ gives rise to a line bundle $$\mL_q(\vphi): = \bC^*\times \bC/(u,v)\sim (uq,\vphi(u)v).$$
In fact, all line bundles on $\Eq$ can be described in this way and $\mL(\vphi)\cong \mL(\vphi')$ if and only if there exists some holomorphic $B:\bC^*\to \Gl(\bC)$  such that $\vphi'(u) = B(qu)\vphi(u)B\inv(u)$.  To see this, note that, as all line bundles on $\bC^*$ are trivial, there exists a global frame on the pullback $\pi^*\mL(\vphi)$.  Given such a frame, $\vphi$ can be thought of as the induced fiber map $(\pi^*\mL(\vphi))_u\to (\pi^*\mL(\vphi))_{qu}$.  If $\mL(\vphi)\cong\mL(\vphi')$ then there must exist some $\hat B:\mL(\vphi)\arriso \mL(\vphi')$ such that the following diagram commutes.\\
$$\xymatrix{
(\pi^*\mL(\vphi))_u\ar[d]_{\pi^*\hat{B}(u)} \ar[r]^{\vphi(u)} &(\pi^*\mL(\vphi))_{qu}\ar[d]^{\pi^*\hat{B}(qu)}\\
(\pi^*\mL(\vphi'))_u \ar[r]^{\vphi'(u)} &(\pi^*\mL(\vphi'))_{qu}
}$$

The fact that any such $B$ descends to and isomorphism of line bundles on $\Eq$ can be shown by simply verifying that the line bundle $\mL(\vphi)$ is well-defined for any $\vphi$.\\
Below we will often make use of the fact that $\mL(\vphi_1)\otimes \mL(\vphi_2) = \mL(\vphi_1\vphi_2)$.  Recalling that the tensor product of two line bundles gives a line bundle whose transition functions are given by the product of the respective transition functions of those bundles, this fact is clear.\\


Consider in particular $\Lp{}:=\mL(\varphi_0)$, where $\varphi_0(u) = q^{-1/2}u\inv$. Notice that, by \ref{thm:thetaprops} above, $\theta_q(z)$ descends to a holomorphic section of $\mL$.  This is our special choice of degree one line bundle to parameterize all others.\\
Recall that, for positive degree line bundles on the elliptic curve, $h^0(\mL) = \deg(\mL)> 0$.\\
For $n>0$, we similarly have that a basis of $H^0(\mL^n)$ is given by $\theta[j/n,0](n\tau,nz)$ for $j\in \bZ/n\bZ$.\\
From the above parameterization of the Picard group, it is also clear that we have $t^*_x\Lp{} = \mL(t^*_x\vphi_0)$.  Together, these facts allow us to completely describe the additive structure of the category of line bundles on an elliptic curve.\\


In particular, let $\mL_1 = (t_x^*\Lp{})\otimes\Lp{n-1}$ and $\mL_2 = (t_y^*\Lp{})\otimes\Lp{m-1}$ be any two line bundles over $\Et$.\\
Then

$$\Hom(\mL_1,\mL_2) = H^0(\mL_1^*\otimes \mL_2) = H^0((t_x^*\Lp{*}\otimes t_y^*\Lp{}\otimes\Lp{m-n})$$

\noindent Using $t_\delta^*\vphi_0^k = e^{-2 \pi i \delta k}\vphi_0^k$, we can rewrite this bundle as
$$t_x^*\Lp{*}\otimes t_y^*\Lp{}\otimes\Lp{m-n} = \mL(t_x^*\vphi_0\inv\cdot t_y^*\vphi_0\cdot\vphi_0^{m-n}) = \mL(e^{2\pi i (x-y)} \vphi_0^{m-n}) = t^*_\delta \Lp{k}$$

\noindent for $k := m-n$ and $\delta := \dfr{y-x}{k}$ and where the last equality holds only if $k\neq 0$.\\

If $\mL_1$ and $\mL_2$ are of the same degree (i.e. $k=0$), then any morphism between them corresponds to a global section the degree zero line bundle $\mL_1^*\otimes \mL_2$.  Non-zero sections of degree zero line bundles are non-vanishing and thus correspond to an isomorphism.  This gives us a contradiction unless $x=y$.  Recalling that only positive degree bundles admit sections, the $\Hom(\mL_1,\mL_2)$ is non-zero if and only if $m>n$ or both $n=m$ and $x=y$.\\

%

To summarize, the only nontrivial case we have is when $k := m-n>0$ and $\delta := \dfr{y-x}{k}$.  In which case we have, 
$$\Hom(\mL_1,\mL_2) = H^0(t_\delta^*\Lp{k}) = \text{span}_\bC\{\theta[j/k,k\delta](k\tau,kz)\}_{j\in \bZ/k\bZ}$$\\

\begin{rem}
One can check\footnote{This means simply checking $t^*_x$ takes homsets to homsets in the obvious way and commutes with the composition (product) of theta functions.} in fact that $t^*_x$ is functorial and thus defines a family (parameterized by the torus) of functorial auto-equivalences of the category of line bundles on $E_\tau$.
\end{rem}

Let us introduce the less cumbersome notation $f_j^{(k)}(z) := \theta[j/k,0](k\tau,kz)$.\\

\section{The A-side}\label{sec:Elliptic}
\bigskip\bigskip
\subsection{B-fields and the \kah{} Moduli Space}\hfill\\

Let $(X,J)$ be a complex manifold and define
$$\cK(X,J):=\{[\omega]\in H^2(X,\bR)\stb \omega \text{ is \kah{}}\}$$
This is usually called the \emph{\kah{} cone} of $X$.  Mirror symmetry suggests that we should consider the \kah{} moduli space, of $(X,J)$, as (the complex manifold):
$$\cM_{\text{\kah{}}}(X,J):=(H^2(X,\bR) \oplus i\cK(X,J))/H^2(X,\bZ)$$
A closed 2-form $\omega_\bC = B+i\omega$ representing an element in this space is called a \emph{complexified \kah{} form} and this tacked on closed 2-form $B$ is called a \emph{B-field}.\footnote{The name for this physically motivated structure comes from an analogy to denoting magnetic fields with a $B$ and not from reference to Type IIB string theory (the B-side of mirror symmetry).  The mathematical meaning of B-fields has been studied by Hitchin and others in the newly developing field of ``generalized geometry."}\\

As $H^2(\Et,\bR) = \bR$, our complexified \kah{} form is determined by a complex number $\rho = B+iA$, with $A>0$.  We will from now on denote the (real) 2-torus equipped with this complexified \kah{} form by $\Er$.

\begin{rem}Homological mirror symmetry is a relationship only conjectured to exist between a \kah{} manifold and its mirror dual.  That said, the Fukaya category will only depend on the symplectic structure and the derived category of coherent sheaves will only depend on the complex structure (of the mirror dual).\end{rem}

\bigskip\bigskip
\subsection{The Fukaya Category}\hfill\\

It is not clear exactly what the objects of the Fukaya category should be.  Fukaya defined the category \cite{FuII,FOOO2009} over (a countable set of) closed Lagrangian submanifolds equipped with line bundles of curvature $B|_L$.  Kontsevich, inspired by string theoretic D-branes, suggested taking these objects as closed special Lagrangian submanifolds equipped with flat complex line bundles with unitary connections.  Also, it is, in general, necessary to add some additional structure for purposes of grading and further require that our Lagrangian submanifolds are relatively spin. This latter condition is automatically satisfied in the one-dimensional case - the (relative) spin structure is a $\bZ_2$ choice here ($H^1(L;\bZ_2)$ in general) and can be suppressed.\\

In the case of elliptic curves, each Hamiltonian isotopy class of closed Lagrangian submanifolds has a unique special Lagrangian representative.  This can be shown using the Ricci flow and in fact, using our mirror functor to suggest a definition of ``stable" special Lagrangian objects, this given a symplectic analogue of Atiyah's classification of vector bundles on an elliptic curve \cite[section 38.4]{clayMS}.  For a general Calabi-Yau manifold, it was shown in \cite{TY2002} that this uniqueness (but not necessarily existence) is always true when the Fukaya category is unobstructed.  As shown in section \ref{sec:SL}, the closed special Lagrangian submanifolds of $\Ets$ are exactly the geodesics (i.e. the projections, through $\bC\to\Ets$, of rationally sloped lines).  Following the middle-ground taken by \cite{PZ,Kreuss}, we will equip our Lagrangian submanifolds with a local systems whose monodromy has eigenvalues of unit modulus (i.e. complex vector bundles equipped with connections whose holonomy group is generated by fiber automorphisms with eigenvalues of unit modulus).  The Jordan blocks will be related to non-stable vector bundles on the B-side.\\

\bigskip\bigskip
\subsection{The Objects}\hfill\\

In specific, here we define the objects of $\Fuk(\Er)$ to be the graded, oriented, closed geodesics (of $\Er$) equipped with a flat complex vector bundle of monodromy with unit eigenvalues (in the sense discussed in remark \ref{rmk:monodromy}).
All Lagrangian submanifolds of a 2-surface are of course homeomorphic to $S^1$ and, as shown in section \ref{sec:VB}, all complex vector bundles on $S^1$ are trivial.  Moreover, we know from corollary \ref{thm:S1con}, that any flat connection on a complex vector bundle over $S^1$ is determined (up to gauge equivalence) by the conjugacy class (in $\Gl_k(\bC)$) of its monodromy operator $M$.\\

We will give a grading to our special Lagrangian submanifolds, $L$, determined by a choice $\alpha\in \bR$ such that $z(t) = z_0 + te^{\pi i \alpha}$ pararameterizes a lift of $L$ to $\bC$ (for some appropriate $z_0$).\\
$\alpha$ also determines an orientation on $L$ given by the direction of $e^{\pi i \beta}$ for the unique $\beta \in (-1/2,1/2]$ such that $\alpha - \beta \in \bZ$.\footnote{This grading and orientation match those constructed in appendix \ref{sec:phaseL}.}\\

So, following the notation in \cite{Kreuss}, we will use tuples $\sL=(L,\alpha,M)$ to denote the objects of $\Fuk(\Er)$.  And we will often use $M_p$ to refer to the stock of our local system (i.e. fiber of our vector bundle) over a point $p\in L$.
We have a shift functor, which mirrors that in $\dbcoh(E_\rho)$, given by
$$(L,\alpha,M)[1]: = (L,\alpha+1,M))$$

\bigskip\bigskip
\subsection{The Morphisms}\hfill\\

Suppose that $\sL$ and $\sL'$ are two objects such that $L\neq L'$.
$$\Hom(\sL,\sL'): = \bigoplus_{p\in L\cap L'}\Hom(M_p,M'_p)$$
Similarly if $L=L'$, we can define
$$\Hom(\sL,\sL'): = \Hom(M,M'),$$
where $\Hom(M,M')$ is the space of vector bundle morphisms between the associated complex vector bundles (modulo isomorphisms of $L$ and $L'$).\\
We could of course write these two cases together as 
$$\Hom(\sL,\sL'): = \Hom(\rest{M}_{L\cap L'},\rest{M}_{L\cap L'})$$

A $\bZ$-grading is put on morphisms $\sL\to \sL'$ by the Maslov-Viterbo index, which in this case is given by $$\label{eq:MVindex}\mu(\sL,\sL') = \text{ceil}(\alpha - \alpha')= -\text{floor}(\alpha' - \alpha)\in \bZ$$

\bigskip\bigskip
\subsection{The \Ainf{}-structure}\hfill\\

The Fukaya category is not a true category in the sense that the composition of morphisms is not associative.  It does, however, admit an \Ainf{}-structure.\\

Let $R$ be some commutative ring (for mirror symmetry, we typically take $\bC$ or $\bQ$).\\ 
An \emph{$A_\infty$ category} is given by a set of objects $\ob$, a graded free $R$-module $Hom(c_1,c_2)$ for each $c_1,c_2\in \ob$, and a family of degree $2-k$ operations
$$\fm_k:Hom(c_0,c_1)\otimes...\otimes Hom(c_{k-1},c_k)\to Hom(c_0,c_k)$$
which satisfy the ``$A_\infty$ associativity relations" for $k\geq 1$:\\

$$\sum_{r=1}^{n}\sum_{s=1}^{n-r+1}(-1)^\epsilon \fm_{n-r+1}(a_1\otimes...\otimes a_{s-1}\otimes \fm_r(a_s\otimes...\otimes a_{s+r-1})\otimes a_{s+r}\otimes...\otimes a_n) = 0$$
for all $n\geq 1$, where $\epsilon = (r+1)s + r(n+\sum_{j=1}^{s-1}\text{deg}(a_j))$.

From $n=1$, we have $\fm_1^2 = 0$ and giving us a differential.\\
From $n=2$, we have $\fm_1(\fm_2) = \fm_2(\bid\otimes \fm_1) + \fm_2(\fm_1\otimes \bid)$, which tells us that $\fm_2$ defines a multiplication operator satisfying the Leibniz rule of $\fm_1$.\\
From $n=3$, we have that $\fm_2(\fm_2\otimes\bid)-\fm_2(\bid\otimes \fm_2) = \fm_1(\fm_3) + \fm_3(\fm_1 \otimes \bid\otimes\bid) + \fm_3(\bid\otimes \fm_1\otimes\bid) + \fm_3(\bid\otimes\bid\otimes \fm_1)$, which tells us that $\fm_2$ is associative ``up to a homotopy" $\fm_3$ (in the sense described for chain maps in appendix \ref{sec:chaincomp}).  In particular this third condition says $\fm_2$ is associative at the level of cohomologies.\\

\begin{rem}
In our case, the Fukaya category on an elliptic curve, we will have that $\fm_1=0$ and thus $\fm_2$ will be associative.  Generally speaking the Fukaya category is not associative.  Typically when constructing the Fukaya category we also need to use a \emph{curved} (also called \emph{obstructed}) $A_\infty$ structure.  This means we add a map $\fm_0$ (and start the sum with $r=0$) obstructing our differential category structure.  For example, the first two relations would then be $\fm_1(\fm_0) = 0$ and $\fm_1^2 + \fm_2(\fm_1\otimes \bid) + \fm_2(\bid\otimes \fm_0) = 0$.  This obstruction, $\fm_0$, is mirror to the Landau-Ginzberg superpotential on the B-side \cite{ArrAntican}.
\end{rem}

The \Ainf{}-structure of the Fukaya category is given by summing over\\
pseudo-holomorphic disks bounded by Lagrangian submanifolds.  To be specific, we will need to define a particular collection of $k+1$-pointed (pseudo-)holomorphic disks.  While we do this we will use an index $j\in \bZ/\bZ_{k+1}$.\\

\begin{rem}
The Fukaya category generally involves pseudo-holomorphic disks, but as every almost complex structure on a surface is integrable, we can drop the pseudo in our one dimensional case.\end{rem}

Let ${\sL_j}$ be a collection of $k$ objects and fix a point $p_j\in L_j\cap L_{j+1}$.  As we are using the index $j\in \bZ/\bZ_{k+1}$, $p_{k}\in \sL_{k}\cap\sL_0$.  Let $D:=\{z\in \bC \stb z\leq 1\}$ and let $(D,S_{k+1})$ denote a disk with $k+1$ marked points $S_k = \left\{e^{i\theta_j}\right\}\subset \partial D$ such that $0=\theta_0<\theta_1<...<\theta_{k}< 2\pi$.\\

$\cat{CM}_{k+1}(X;L_0,...,L_{k};p_0,...,p_{k}):= $
$$\left\{
(\phi,S_k) \left| \begin{split}&\phi:D\to X \text{ is pseudo-holomorphic, }\\
&\phi(e^{i\theta_j}) = p_j,\: \phi(e^{it}) \in L_j \:\forall\, t\in (\theta_{j-1}, \theta_{j}) \an \forall j\end{split}\right.\right\}/\sim
$$
Where $(\phi,S_{k+1})\sim(\phi',S_{k+1}')$ if there exists a biholomorphic map $f:D\to D$ such that $\phi'\circ f = \phi$ and $f(e^{i\theta_j}) = e^{i\theta_j'}\:\forall\,j$.\\\\

\begin{rem}The ordering of $\theta$'s ensures that as we traverse the boundary of the disk, we hit the marked points in order.  Parallel transport along this boundary will define our \Ainf ``composition maps".\end{rem}  
%

Let $u_\ell\in \Hom(\sL_{\ell-1},\sL_{\ell})$ for $\ell = 1,..,k$ be a collection of $k$ morphisms and assume the involved $k+1$ Lagrangian submanifolds are distinct.  We will often refer to this as \emph{the transversal case}.\\

For notational convenience, we will let morphisms be denoted by sums of the form $\sum_\mu t_\mu\cdot p_\mu$ where each $p_\mu$ is an intersection point and each $t_\mu$ is a morphism between fibers over $p_\mu$.\\
$$\fm_k(u_1\otimes...\otimes u_k) = \sum_{p\in L_{k}\cap L_0}C(u_1,..., u_k;p)\cdot p$$ where the fiber morphisms are defined by
$$C(u_1,..., u_k;p) = \sum_{[\phi]} e^{2\pi i \int\phi^*\omega_\bC}h_{\pd \phi}(u_1\otimes...\otimes u_k)$$

This second sum is over equivalence classes $[\phi]$ in the set\\
$\in \cat{CM}_{k+1}(\Er;L_0,...,L_k;p_0,...,p_{k-1},q)$ and
$$h_{\pd \phi}(u_1\otimes...\otimes u_k): = P_{\gamma_k}\circ u_k\circ...\circ u_1 \circ P_{\gamma_0},$$
where $P_{\gamma_j}:(M_j)_{p_{j-1}}\to (M_j)_{p_{j}}$ is the parallel transport operator (induced by the connection of $\sL_j$) over the curve $\gamma_j(t) = \phi(e^{it})$, $t\in [\theta_{j-1},\theta_j]$.\\

We will not check that these satisfy the \Ainf relations here.  We will give some discussion of the non-transversal case, but only after descending to the zeroth cohomology.  For a more complete discussion of this category in the case of elliptic curves (and more generally for abelian varieties) see \cite{FuAbel}.

\begin{rem}
The construction of the Fukaya category on a general manifold has many technical obstructions.  The relevant obstruction theory and complete definition can be found in \cite{FOOO2009}.  We mentioned previously the addition of a (relative) spin structure to our Lagrangians, which is to solve a problem of orientability of the moduli space of disks.  Another difficulty in higher dimensions is the existence of pseudo-holomorphic bubbling.  This is typically solved by replacing the complex coefficients of our homsets with a universal Novikov ring (where the above $\fm_k$ sums converge).  There is an alternative approach given by cluster algebras in \cite{CL}.
\end{rem}

The following basic property of topological covering spaces will help us count holomorphic discs.\\ 
\begin{lem}[The Lifting Criterion]Let $\tilde X\xrightarrow{p}X$ be any covering space and $\varphi:Z\to X$ a map to its base from a path-connected and locally path-connected space $Z$.  Fix a base point $z\in Z$ and a corresponding fiber point $\tilde x\in p\inv(\varphi(z))$.\\
There exists a lift $\tilde \varphi:Z\to \tilde X$ of $\varphi$ such that $p\circ\tilde \varphi(z)= \varphi(z)$ if and only if the induced homomorphisms $\varphi_\sharp :\pi_1(Z,z)\to \pi_1(X,\varphi(z))$ and $p_\sharp :\pi_1(\tilde X,\tilde x)\to \pi_1(X,\varphi(z))$ are such that the image of the former is contained in the image of the later.  I.e.
$$\varphi_\sharp(\pi_1(Z,z))\subset p_\sharp(\pi_1(\tilde X,\tilde x))$$
Moreover, if such a lift $\tilde \varphi$ exists, it is unique.\end{lem}
\begin{proof}This result can be easily gotten from the homotopy lifting property.  A detailed proof can be found in \cite{HatcherAT}.\end{proof}

\begin{rem}\label{rmk:kgon}Notice that if $Z$ is simply connected (e.g. a disc), this condition is trivially satisfied.  Thus we have that, any mapping of a $k$-pointed disk into $\Er$ (like those used to define $\fm_k$) lifts to a $k$-gon in $\bC$ with sides lying on lines lifted from $k$ geodesics.  As $\pi:\bC\to \Ert$ is a holomorphic map, we know, by the Riemann mapping theorem, that each $k$-pointed disc in $\bC$ represents such a map.  This is well-defined as the lifting operation is equivariant with respect to automorphisms of the disk (i.e. $\widetilde{\phi\circ f} = \tilde\phi\circ f$).  This gives us an easy way to count these holomorphic disks and also tells us that $\fm_1=0$.  As $\fm_1$ acts as the differential operator in this category, this means our cohomology complex is identical to our complex of homsets graded by the Maslov-Viterbo index \ref{eq:MVindex}.
\end{rem}

\begin{rem}\label{rmk:S1con}
From corollary \ref{thm:S1con}, we can assume (as our Lagrangians are just circles) that the above parallel transport operators $P_\gamma$ are given by $P_\gamma = \exp(M\ell)$, where $M$ is the monodromy operator and $\ell$ is the length of $\gamma$ divided by the length of the Lagrangian (e.g. if $\gamma$ covers the Lagrangian, this is the winding number).
\end{rem}

\section{The Simplest Example}\label{sec:simple}

As we will see in the following sections, the heart of homological mirror symmetry for elliptic curves is a functor between holomorphic vector bundles (and skyscraper sheaves) over $\Et$ and complex local systems over non-vertical (and, respectively, vertical) geodesics of $\Ets$.  Let us take a look at how holomorphic line bundles will be sent to lines of integer slope.  More specifically, our functor takes gives the following correspondence between objects.

\begin{center}
  \begin{tabular}{ c | c }
  A-side Objects                         & B-side Objects\\
                                      \hline
rank 1 local systems of                  & holomorphic line bundles\\
slope $d$,                               & \\
$y$-intercept\ftmark{} $y_0$, and                 & $(t^*_{-y_0\tau + \beta}\Lp{})\otimes\Lp{d-1}$\\
connection $\nabla = d -2\pi i \beta dx$ & 
  \end{tabular}
\end{center} 

Except in the horizontal case, we could write the corresponding holomorphic line bundle in terms of the $x$-intercept, $x_0$, as $(t^*_{dx_0\tau + \beta}\Lp{})\otimes\Lp{d-1} = t^*_{x_0\tau + \frac{\beta}{d}}\Lp{d}$.

So what about morphisms?  Two geodesics of respective slopes $n_1$ and $n_2$ with respective $x$-intercepts $x_1$ and $x_2$ have intersections at points 
$$e_k = \left(\frac{n_1x_1 - n_2x_2 + k}{n_1-n_2},\frac{n_1}{n_1-n_2}[n_2(x_1-x_2)+k]\right) \:\text{ for } k\in \bZ/(n_1-n_2)\bZ$$
We then want to send these points to linear combinations of morphisms\\
$t^*_{\delta_{21}\tau + \beta_{21}} f^{(n_2-n_1)}_k$, for $\delta_{21} := \frac{n_2x_2-n_1x_1}{n_2-n_1}$ and $\beta_{21}:= \frac{\beta_2 - \beta_1}{n_2-n_1}$.\\

Our functor will be given by 
$$t^*_{\delta_{21}\tau + \beta_{21}} f^{(n_2-n_1)}_k \mapsto e^{\pi i \tau (n_1-n_2)\delta_{21}^2}\cdot e_k $$

\begin{rem}
Note that we can not simply think of this as sending theta functions to points or vise versa.  We will give below an example where multiple theta functions are sent to the same point (but each time that point is thought of as living in a different homset).  That said, $t^*_{\delta\tau+\beta}f^{(n)}_k$ is always going to be sent to a point with $x$-coordinate $\delta + k/n$ and coefficient $e^{-\pi i n\tau\delta^2}$.  If we imagine translating one Lagrangian in the $x$-direction while keeping the other fixed, we see this translation must correspond to a change in $\delta$ of the same amount.\\

We only have one direction to translate in on the A-side, but we can also twist our line bundle, which we see corresponds to changing $\beta$ by the proportional amount.  The monodromy (i.e. twisting) can be easily visualized by thinking of $L \times \bC\cong S^1\times D$ as a solid donut containing a curve (not intersecting the zero curve) representing the parallel transport of the unit once around the circle.  We can picture the rank $k$ situation by imagining $k$ (ordered) curves on each of $k$ solid donuts, but it is difficult to visualize the nondegeneracy for $k>>1$.

\end{rem}  


To motivate this correspondence, let us go through the simplest nontrivial example.\footnote{This example was first given in \cite{PZ}.}  Let $\tau = iA \in i\bR$ and consider three Lagrangian submanifolds $L_0,L_1,L_2$, whose lifts to $\bR^2$ can be chosen to all travel through the origin, and have slopes $0,1, \an 2$ respectively.  For now lets assume the trivial connection $\nabla = d$.  The grading (i.e. choice of $\alpha$) will not yet be important, and, for now, we are assuming the trivial connection, so we do not need to worry about anything but the submanifold itself.  The utility of the grading is mostly in finding a categorical equivalence (we need a shift functor on the A-side and the grading will play that role).  In the simple example we are discussing here, we will only need the grading for the orientation it endows on our Lagrangian submanifolds.  For most of this simple example, we will assume trivial connections, which will allow us to suppress the choice of grading all together.\\

\begin{figure}[H]
\label{fig:SimpleTriangle}
\includegraphics[width=1.8in,hiresbb=true]{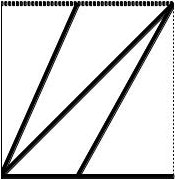}
\caption{The bold lines mark our three Lagrangian submanifolds.  Dotted lines signify repetitions.}
\end{figure}
On the B-side, these objects will correspond to line bundles $\mL_0 = \cO$, $\mL_1 = \Lp{}$, and $\mL_2 = \Lp{2}$.  Note $\Lp{} := \mL(\varphi_0)$ is our special choice of the degree 1 line bundle defined above.\\
We then want our functor to send
$\bC^{|\sL_i\cap \sL_j|} = \Hom(\sL_i,\sL_j)\to\Hom(\mL_i,\mL_j) = H^0(\mL_j\otimes \mL_i^*)$ in a way that is compatible with composition.\\

Observe that\\
$\Hom(\mL_0,\mL_1) \cong H^0(\Lp{}) = \text{span}_\bC(\{\thu(z)\})$\\
$\Hom(\mL_1,\mL_2) \cong H^0(\Lp{}) = \text{span}_\bC(\{\thu(z)\})$\\
$\Hom(\mL_0,\mL_2) \cong H^0(\Lp{2}) = \text{span}_\bC(\{\theta[0,0](2\tau,2z), \theta[1/2,0](2\tau,2z)\})$\\

Composition of morphisms is given by fiberwise multiplication.  So the above addition formula, \ref{thm:thetaadd}, determines all compositions on the A-side of this simple example.\\

Let us take a look at the A-side and find our $\fm_2$ composition operator.  From figure \ref{fig:SimpleTriangle}, we can see\\
$L_0\cap L_1 = \{e_1\}$\\
$L_1\cap L_2 = \{e_1\}$\\
$L_0\cap L_2 = \{e_1,e_2\}$\\

Recall our lifting lemma (and remark \ref{rmk:kgon}) above.  Fixing the lift of $e_1$ to be the origin in $\bC$, our triangles are determined by the winding number (including orientation\ftmark{}) of the boundary segment mapped to $L_1$.\\
\fttext{With respect to the orientation given to $L_1$ by its grading (i.e. choice of $\alpha$).}
 
Notice that when immersion $\gamma_1\imm L_1$ has an even winding number, all corners of the triangle descend to the same point in $\Et$.  In the odd case, $e_2 = \pi(1/2)$.  The areas\ftmark{} of these triangles respectively are $\phi^*\omega_\bC = \tau n^2 = iAn^2$ and $\phi^*\omega_\bC = \tau n^2 = iA(n+1/2)^2$, where $n$ is the winding number of $\gamma_1\imm L_1$.\\
\fttext{In this case with trivial B-field, $\phi^*\omega_\bC$ gives the area of the lifted disk, but in general this quantity will be complex and is often referred to as the \emph{energy} of $\phi$.}

$$\fm_2(e_1,e_1) = C(e_1,e_1,e_1)\cdot e_1 + C(e_1,e_1,e_2)\cdot e_2$$
Where
\begin{align*}
&C(e_1,e_1,e_1) = \sum_{n\in\bZ}e^{-2\pi An^2}\\
&C(e_1,e_1,e_2) = \sum_{n\in\bZ}e^{-2\pi A(n+1/2)^2}
\end{align*}
Notice that $C(e_1,e_1,e_1) = \theta[0,0](2iA,0) = \theta_{2\tau}(0)$ and $C(e_1,e_1,e_2) = \theta[1/2,0](\tau,0)$.\\

Thus $\fm_2(e_1\otimes e_2) = \theta_{2\tau}(0)\cdot\theta_{2\tau}(2z) +\theta[1/2,0](2\tau,0)\cdot\theta[1/2,0](2\tau,2z)\in H^0(\Lp{2}) = \Hom(\mL_0,\mL_2)$.\\
In fact, by the addition formula above, $\fm_2(e_1\otimes e_1) = [\thu(z)]^2$.\\

This suggests our functor might send
\begin{align*}
\Hom(\sL_0,\sL_1)&\to \Hom(\mL_0,\mL_1)\\
e_1&\mapsto \thu(z)\\
&\\
\Hom(\sL_1,\sL_2)&\to \Hom(\mL_1,\mL_2)\\
e_1&\mapsto \thu(z)\\
&\\
\Hom(\sL_0,\sL_2)&\to\Hom(\mL_0,\mL_2)\\
\left\{\begin{array}{r}e_1\\e_2\end{array}\right.
&\begin{array}{l}\mapsto \theta[0,0](2\tau,2z)\\\mapsto \theta[1/2,0](2\tau,2z)\end{array}
\end{align*}

Let us now complicate our simple example a bit and see what happens on the A-side when we replace $L_2$ with another geodesic of slope 2.  Let $L_2$ be the geodesic in $E^\tau$ descending from the line of slope 2 with $x$-intercept given by $x_0\in (0,1/2)$.\footnote{We can assume $x_0\in (0,1/2)$ without any loss of generality.}  Let $\sL'_2$ denote the corresponding local system with trivial connection.\\
\begin{figure}[H]
\label{fig:FourTrianglesShifted}
\includegraphics[width=1.8in,hiresbb=true]{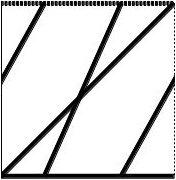}
\caption{The shifted case.}
\end{figure}
Now we have
$$\fm_2(e_1\otimes e'_2) = C(e_1,e'_2,p_1)\cdot p_1 + C(e_1,e'_2,p_2)\cdot p_2$$
Where
\begin{align*}&C(e_1,e'_2,p_1) = \sum_{n\in \bZ}e^{-2\pi A(n+x_0)^2} = \theta[x_0,0](2\tau,0)\\
&C(e_1,e'_2,p_2) = \sum_{n\in \bZ}e^{-2\pi A(n + 1/2 - x_0)^2} = \sum_{n\in \bZ}e^{-2\pi A(n+x_0 + 1/2)^2} = \theta[1/2 + x_0,0](2\tau,0)\end{align*}
\noindent
Note that $\theta[1/2 - x_0,0](2\tau,0) = \theta[1/2 + x_0,0](2\tau,0)$.\\

\noindent
According to the above functor, the mirror of $\sL'_2$ should be $t^*_{x_0\tau}\Lp{2}$.  So in this situation we have
\begin{center}\begin{align*}
&\Hom(\mL_0,\mL_1) = \spanc{\thu(z)}\\
&\Hom(\mL_1,\mL'_2) = H^0(t^*_{2x_0\tau}\Lp{}) = \spanc{t^*_{2x_0\tau}\thu(z)}
\end{align*}\end{center}

\noindent
Thus our compositions $\mL_0\to \mL_1\to\mL'_2$ are determined by the product\\
$\thu(z)\cdot \thu(z+2x_0\tau)$ and should be members of the space\\
$$\Hom(\mL_0,\mL'_2) = H^0(t^*_{x_0\tau}\Lp{2}) = \spanc{\theta[j/2,2x_0\tau](2\tau,2z)}_{j\in \bZ/j\bZ}$$

\noindent
From the above addition formula for theta functions, we have
\begin{align*}
\thu(z)&\cdot \thu(z+2x_0\tau)\\
&= \theta_{2\tau}(2x_0\tau)\theta_{2\tau}(2z+2x_0\tau) + \theta[1/2,0](2\tau,2x_0\tau)\theta[1/2,0](2\tau , 2z+2x_0\tau)\\
&= e^{-2\pi i x_0^2\tau}\left[\theta[x_0,0](2\tau,0)\theta[0, 2x_0\tau](2\tau,2z)\right. \\
&\left.  \hspace{1.9 in} +\theta[x_0 + 1/2,0](2\tau,0)\theta[1/2,2x_0\tau](2\tau,2z)\right]
\end{align*}

\noindent
This tells us that, our functor preserves composition by sending
\begin{align*}
\Hom(\sL_0,\sL_2)&\to\Hom(\mL_0,\mL_2)\\
\left\{\begin{array}{r}p_1\\p_2\end{array}\right.
&\begin{array}{l}\mapsto \theta[0,0](2\tau,2z)\\\mapsto \theta[1/2,0](2\tau,2z)\end{array}
\end{align*}

Before moving on from this simple example, let us see what happens when we give $\sL_2$ a nontrivial connection.  Let $\nabla_2 = d + 2\pi i \beta dt_2$ for $t_2$ some coordinate on $L_2$ with $t_2+1~t_2$.  Then

\begin{align*}
&C(e_1,e_1,e_1) = \sum_{n\in\bZ}e^{-2\pi An^2 + 2\pi i n \beta} = \theta[x_0,\beta](2\tau,0)\\
&C(e_1,e_1,e_2) = \sum_{n\in\bZ}e^{-2\pi A(n+1/2)^2 + 2\pi i (n+1/2) \beta} = \theta[x_0 + 1/2,\beta](2\tau,0)
\end{align*}

Again we find the product on the B-side using equation \ref{eq:thetaprodex}; this time with $x = 2x_0\tau + \beta$.  It will be left as an exercise to the reader to check that composition is again preserved in this case.\\

\section{The B-side}\label{sec:EllDb}
\bigskip\bigskip
\subsection{The Derived Category of Coherent Sheaves}\hfill\\

Let $\cA$ be some abelian category and consider, $\cC\cA$ the \emph{category of chain complexes} in $\cA$, where objects are chain complexes (of objects in $\cA$) and morphism are chain maps.\ftmark{}  The \emph{homotopy category} of $\cA$, denoted $\cat{HA}$, has the same objects as $\cA$, but we consider the morphisms up to homotopy equivalence.  The \emph{derived category} of $\cA$, $\cat{DA}$, is the localization of $\cat{HA}$ over quasi-isomorphisms (i.e. we formally add inverses to quasi-isomorphisms).  In mirror symmetry we always talk about the bounded derived category of coherent sheaves on some variety $X$, $\dbcoh(X)$ (or often, as we will do here, simply denoted $\db(X)$).  This is the full subcategory of $\cat{DA}$ given by chain complexes in $\coh_X(A)$ of finite length.  These concepts are explained in more detail in section \ref{sec:DerivedCat}, but in the case we are interested in here, $\db(\Et)$, little understanding of this general topic is needed.  In this section we will collect some standard facts about $\db(X)$ and coherent sheaves on $X$, which we will use to gain a simple description of $\db(\Et)$.  All results in this section are either standard or proven in \cite{PZ,Kreuss}.\\\\
\fttext{Note that by bounded we mean complexes $A^\bullet$ such that $A^n = 0$ for some $n>>0$.  Also, when defining the $\dbcoh(X)$, we technically use cochain complexes.}

The classic example of a coherent sheaf is the space of local sections of a vector bundle.  This correspondence gives a natural equivalence between locally free coherent sheaves and vector bundles over manifolds.  With this in mind, we will often, in a slight abuse of language, call a (locally free) coherent sheaf a vector bundle or visa-versa.  To avoid unnecessary (in our simple case of elliptic curves) algebraic geometry we will offer the following theorem in lieu of the definition of a \emph{coherent sheaf}.\\

\begin{thm}All indecomposable\ftmark{} coherent sheaves over a Riemann surface are either an indecomposable vector bundle or a torsion sheaf supported at one point.\end{thm}\begin{proof}See \cite{PZ} or \cite[Ex. III.6.11]{HartAlgGeo}.\end{proof}
\fttext{A sheaf (or vector bundle) $X$ is \emph{indecomposable} if $X\cong X_1 \oplus X_2 \rimp X_1\cong 0 \text{ or } X_2\cong 0$.}

\begin{rem}A torsion sheaf, $\mF$, on $\Et$ supported at a single point, $\zeta_0\in \Et$, is determined by a finite dimensional (complex) vector space $V:=\mF_{\zeta_0}$ and a nilpotent endomorphism $N\in End(V)$.  Following \cite{PZ,Kreuss}, we will denote such an sheaf by $S_\tau(\zeta_0,V,N)$.  In other words, 
$$S_\tau(\zeta_0,V,N):=\cO_{r\zeta_0}\otimes V/<\zeta - \zeta_0 - \frac{1}{2\pi i} N>$$  
where $\cO_{r\zeta_0}$ is defined by the exact sequence 
$$\cO\xrightarrow{(\zeta-\zeta_0)^r}\cO\to \cO_{r\zeta_0}$$.\\

$S_\tau(\zeta_0,V,N)$ is indecomposable if and only if $N$ has a one dimensional kernel (as this implies $N^r=0$ and $N^{r-1}\neq 0$ for $r=\dim(V)$).\end{rem}  

For a object $\mF$ in $\db(C)$, let $\mF[-n]$ denote the chain complex composed entirely of zero objects except for its $n\ith$ degree term, which is $\mF$.\footnote{We will often abuse notation by denote $\mF[0] = \mF$.  Also, if $\mF$ is already a chain in some category of complexes, we will use $\mF[-n]$ to denote the corresponding shifted chain.}\\

\begin{thm}[\cite{Kreuss}]
Let $C$ be a compact Riemann surface.  Finite direct sums of objects $\mF[n]$, where $\mF$ is an indecomposable coherent sheaf on $C$, form a full subcategory of $\db(C)$ which is equivalent to $\db(C)$.\\
\end{thm}

We can state Serre duality on the B-side as follows (see \cite{Kreuss} for a proof),
\begin{lem}\label{lem:Bserre}Let $\mF$ and $\mG$ be to coherent sheaves on $\Et$, then there is a functorial isomorphism $\Hom(\mF,\mG)\cong \ext^1(\mG,\mF)^*$.
\end{lem}

For any $A,B$ objects of any abelian category we have that\footnote{This is often taken as a definition, see \cite{Keller} or \cite[section I.6]{HartResDua} for an explanation.}\\
$$\Hom_{\cat{DA}}(A[m],A[n]) \cong \Hom_{\cat{DA}}(A,B[m-n])\cong \ext^{m-n}(A,B)$$
\noindent
Letting our abelian category be that of coherent sheaves on $\Et$ and using the above lemma this gives us functorial isomorphisms
$$\begin{array}{l}
\Hom_{\db(\Et)}(A,B[1]) = \ext^1(A,B) = \Hom(B,A)^*\\
\Hom_{\db(\Et)}(A,B) = \ext^0(A,B):=\Hom(A,B)
\end{array}$$
\noindent

\begin{lem}For $A,B$ coherent sheaves on a Riemann surface, $\ext^k(A,B)=0$ unless $k\in {0,1}$.\end{lem}\begin{proof}In general we can say that for a complex $n$-dimensional manifold the $\ext^k$ functor vanishes for $k>n$.  See \cite[Ex. III.6.5]{HartAlgGeo}.\end{proof}

We will discuss a mirror to Serre duality in section \ref{sec:ellequiv}.\\


\bigskip\bigskip
\subsection{The Classification of Holomorphic Vector Bundles on $\Et$}\hfill\\

Let $V$ be an finite dimensional (complex) vector space and consider the holomorphic vector bundle on $\Eq$ given by
$$F_q(V,A)=\bC^* \times V/(u,v)\sim (uq,A\cdot v)$$
for any $A\in \Gl(V)$.\\

As all homomorphic vector bundles on $\bC^*$ are trivial, the related argument in section \ref{sec:EllLB} still holds that we can describe all line bundles in this way (uniquely up to $B:\bC^*\to \mathrm{GL}(V)$ such that $A(u) = B(qu)A(u)B(u)\inv$).\\

If $A=\exp(N)$ for $N$ a nilpotent endomorphism\footnote{$\exp(N)$ is always invertible as $\det(e^N) = e^\text{Tr(N)}$.} with a one dimensional kernel, then $F_q(V,\exp(N))$ is indecomposable.  Moreover, we have the below classification (theorem \ref{thm:atiyah}) by Atiyah \cite{Avbe}.  

Consider $r$-fold covering $\pi_r: E_{q^r}\to E_q$.  $\pi_r$ defines a functor through its pullback and pushforward which will commute with our mirror functor (we will discuss this functor and its mirror in more detail later (see section \ref{sec:isog}).\\

\begin{thm}{\label{thm:atiyah}Every indecomposable holomorphic vector bundle on $\Et$ is of the form $\pi_{r*}(\mL_{q^r}(\vphi)\otimes F_{q^r}(V,\exp(N)))$ for some $\vphi$ and some nilpotent $N$ with a one dimensional kernel.}\end{thm}

In the next section we will formally add biproducts on the A-side and so when defining our functor (and equivalence of additive categories), it will be sufficient to define it on these indecomposable bundles.  Moreover, we will have that this equivalence commutes with $\pi_*$ and thus only have to consider objects of the form $\mL(\vphi)\otimes F(V,\exp(N))$.  Once the functorality of $\pi_*$ is explicitly established, the following proposition tells us our homsets can be again be dealt with using bases of theta functions (with constant vector coefficients).

\begin{lem}
Let $\vphi = (t_x^*\vphi_0) \cdot \vphi_0^{n-1}$ for some $n>0$.  Then for any nilpotent $N\in \End(V)$, there is a canonical isomorphism
\begin{align}\label{thm:H0}
\Psi_{\vphi,N}:= H^0(\mL(\vphi))\otimes V &\arriso H^0(\mL(\vphi)\otimes F(V,\exp(N))\\
f\otimes v &\mapsto \exp(DN/n)f\cdot v
\end{align}
Where $D = -u\frac{d}{du} = -\frac{1}{2\pi i} \ddz$. 
\end{lem}\begin{proof}
This is simple to check.  See \cite{PZ}.
\end{proof}

Consider two vector bundles $\Nu_i = \mL(\vphi_i)\otimes F(V_i,\exp(N_i))$ with $\vphi_i = (t^*_{x_i}\vphi_0) \cdot\vphi_0^{n_i-1}$ and $n_1<n_2$, using lemma \ref{thm:H0}, we then have
\begin{align*}
\Hom(\Nu_1,\Nu_2) &= H^0(\mL(\vphi_1\inv \vphi_2)\otimes F(V_1^*\otimes V_2,\exp(\bid\otimes N_2 - N_1^*\otimes\bid)\\
&\cong H^0(\mL(\vphi_1\inv \vphi_2))\otimes \Hom(V_1,V_2)
\end{align*}

\noindent
Note: $F(V,A)^* = F(V^*,(A\inv)^*) \cong F(V,A\inv)$.\\




\section{The Equivalence}\label{sec:ellequiv}
\bigskip\bigskip
\subsubsection{The Fukaya-Kontsevich Category}\hfill\\

The equivalence proven in \cite{PZ,Kreuss} is not that of \Ainf{}-categories, but additive categories $\db(\Et)$ and $\cat{FK}^0(E^\tau):=\underline{H^0(\Fuk(E^\tau))}$.\ftmark{}  Let us now explain what we mean by this A-side category.\\
\fttext{Note that we only defined $\fm_k$ in the case of $k+1$ distinct Lagrangian submanifolds.  \cite{PZ,Po} only deal with this case and \cite{Kreuss} sidesteps the higher \Ainf{} operators, $\fm_k$ for $k\geq 3$, altogether by constructing $\fuk^0$ directly.  A definition of $\fm_k$ in the general case is given by \cite{FuAbel,FOOO2009}, but to the author's knowledge, no proof of equivalence between \Ainf{}- (or triangulated) categories in the general (non-transversal) case has been published.  It is likely, however, this proof is thought implicit in the combined works of \cite{Po,FuAbel}.}

As we showed previously (see remark \ref{rmk:kgon}), the differential operator on $\fuk(E^\tau)$ is trivial.  Thus the cohomology complex is identical to the complex given by homsets as graded by the Maslov-Viterbo index \ref{eq:MVindex}.
First, following \cite{PZ}, lets descend/restrict to the trivial cohomology, i.e. $\fuk^0:=H^0(\Fuk(E^\tau))$, where we have the same objects, but only consider morphisms $(L,\alpha,M)\to (L',\alpha',M')$ such that $\mu_{\sL,\sL'} = 0$ (i.e. $\alpha'-\alpha\in [0,1)$.\\
We only dealt with transversal morphism above, in that case ($L\neq L'$), we have\footnote{We will generally drop the $\fuk^0$ from our homsets when it is clear from context.  Same goes for the B-side.}
$$\Hom_{\fuk^0(E^\tau)}(\sL,\sL') = 
\pws{
0&\text{  if } \alpha' - \alpha\notin [0,1)\\
\bigoplus_{p\in L\cap L'}\Hom(M_p,M'_p)&\text{  if } \alpha' - \alpha\in [0,1)}$$

\noindent
When $L=L'$, we have
$$\Hom_{\fuk^0(E^\tau)}(\sL,\sL') = \pws{
0&\text{  if } \alpha' - \alpha\notin {0,1}\\
H^0(L,\Hom(M,M')) = \Hom(M,M')&\text{  if } \alpha' = \alpha\\
H^1(L,\Hom(M,M'))&\text{  if } \alpha' = \alpha + 1}$$

Descending to $H^0(\fuk)$ leaves us with a true (i.e. associative) preadditive category.\footnote{We will ignore the higher \Ainf{} operators, $\fm_k$ for $k\geq 3$, as we are not seeking an equivalence of \Ainf{}-categories in this weakened case of HMS.}  In order to have equivalence with the entire derived category on the B-side, we must also consider the above degree one morphisms (see \cite{Kreuss} for further details).\\

Recall that in $\fuk^0(E^\tau)$, $(L,\alpha,M)[1]:=(L,\alpha+1,M)$.\\
A proof of the following ``symplectic Serre duality" can again be found in \cite{Kreuss}.
\begin{lem}{\label{thm:sympserre}For $\sL_1,\sL_2$ objects in $\fuk^0(E^\tau)$,  there exists a canonical isomorphism $$\Hom(\sL_1,\sL_2[1])\cong \Hom(\sL_2,\sL_1)^*$$}\end{lem}
\vspace{1 in}
$\Fuk^0(E^\tau)$ is not additive (it is preadditive\ftmark{} by construction).  If (biproduct) $\sL=\sL_1\oplus\sL_2$ exists for some nonzero objects $\sL_k$, then we have nonzero projections and embeddings $p_k:\sL\to\sL_k$ and $i_k:\sL_k\to\sL$.  This is impossible (in the case of distinct Lagrangian submanifolds) as we've restricted to just the zero graded morphisms, so $\Hom(\sL,\sL_1)$ and $\Hom(\sL_1,\sL)$ can not both be non-trivial.  
This is a problem as we are looking to show equivalence with the additive category $\db(\Et)$.  \cite{Kreuss} fixed this by inserting these biproducts formally.  We call this formally enlarged category the \emph{Fukaya-Kontsevich} category, $\cat{FK}^0(E^\tau): = \underline{H^0(\Fuk(E^\tau))}$.\\
\fttext{A category is called \emph{preadditive} (or an \emph{\textbf{Ab}-category}) if its homsets are abelian groups and the composition of morphisms is bilinear.  It is called \emph{additive} if it additionally it has a zero object and contains all finite biproducts.}\\

Let $\cA$ be a preadditive category.  We define $\underline{\cA}$ to be the category whose objects are ordered $k$-tuples ($k\geq 0$) of objects in $\cA$ and whose morphisms are matrices of morphisms in $\cA$.  Composition is given by matrix multiplication and the zero tuple is the zero object.\\

%

\bigskip\bigskip
\subsection{The Mirror Isogenies}\label{sec:isog}\hfill\\

Before we construct the functor between $\db(\Et)$ and $\cat{FK}^0(E^\tau)$, we need to define $p_r:E^{r\tau}\to E^{\tau}$ the mirror of $\pi_r$.\\

Recall that $\pi_r$ is the $r$-fold covering $E_{r\tau}\xrightarrow{\pi_r} \Et$.  The pullback of a vector bundle through $\pi_r$ is given by
$$\pi_r^*F_q(V,A)\to F_{q^r}(V,A^r)$$
and the pushforward is given by
$$\pi_{r*}F_{q^r}(V,A) = F_{q}(V^{\oplus r},\pi_*A)$$
where $\pi_*(A)\in \Gl(V^{\oplus r})$ is sends $(v_1,...,v_r)\mapsto (v_2,...,v_r,Av_1)$.\\

These functors are defined similarly on indecomposable torsion sheaves $$S_\tau(\zeta_0,V,N):=\cO_{r\zeta_0}\otimes V/<\zeta - \zeta_0 - \frac{1}{2\pi i} N>$$

The mirror to this functor, $p_r:E^\tau\to E^{r\tau}$, is induced by the automorphism $(x,y)\mapsto(rx,y)$ of the real plane (for any positive integer $r$).
  Let us look at the restriction of $p_r$ to a Lagrangian submanifold, $L$.  Let $L$ have slope $n/m$ with $\gcd(n,m) = 1$.  As $p_r$ fixes the $y$,  the degree of $\rest{p_r}_L$ is given by $d = \gcd(n,r)$.\footnote{Observe that, assuming $\gcd(n,m)=1$, $L$ wraps around $E^\tau$ $n$ times in the $y$-direction.}  We then must also have that the preimage of a Lagrangian submanifold (of slope $n/(rm)$) has $N = r/d = r/\gcd(n,r)$ connected components.  I.e. $p_r^{-1}(L) = \{L^{(1)},...,L^{(N)}\}$ (we will use this to define the pullback in a moment).\\

The pushforward $p_{r*}:\cat{FK}^0(\Ets)\to\cat{FK}^0(E^{r\tau})$ is then given by

$$p_{r*}(L,\alpha,M)\mapsto (p_r(L),\alpha',p_{r*}M)$$
where $p_{r*}:(v_1,...,v_d)\mapsto (v_2,...v_d,Mv_1)$ and $\alpha'$ is the unique appropriate value in the same interval $[k-1/2,k+1/2)$ that $\alpha$ is contained in (for some $k\in \bZ$).\\


Functorality is not hard to see.
$$\xymatrix{
\Hom(\sL,\sL')\ar[d]_{p_{r*}} &=& \bigoplus_{\tilde x\in L\cap L'}\Hom(M_{\tilde x},M_{\tilde x}')\ar@{^{(}->}[d]\\
\Hom(p_{r*}\sL,p_{r*}\sL')  &=& Z
}$$

Where
$$Z := \bigoplus_{x \in p_r(L)\cap p_r(L')} \Hom\left(\bigoplus_{\tilde x \in (\rest{p_r}_L)\inv(x)} M_{\tilde x},\bigoplus_{\tilde x' \in (\rest{p_r}_{L'})\inv(x)} M'_{\tilde x'}\right)$$\\

Similarly, the pullback $p^*_r:\cat{FK}^0(E^{r\tau})\to\cat{FK}^0(\Ets)$ is given by
$$p_r^*(L,\alpha,M) = \bigoplus_{k=1}^N(L^{(k)},\alpha',(p_r^{(k)})^*M)$$
where $p_r^{(k)}: = \rest{p_r}_{L^{(k)}}$ and $(p_r^{(k)})^*M$ is the local system pulled back in the typical way.\\

A check of the following duality between $p_r^*$ and $p_{r*}$ is given in \cite{Kreuss}.
\begin{prop}\label{isogdual}
Let $p = t'_{n/m}\circ p_r$ for $n/m\in \bQ$ and $t'_\delta(x,y): = (x-\delta,y)$.
$$\Hom(p^*\sL,\sL')\cong \Hom(\sL,p_*\sL')$$
and
$$\Hom(p_*\sL,\sL')\cong \Hom(\sL,p^*\sL')$$
The same is true on the B-side if we replace $p$ by its mirror, $\pi = t_{n\tau/m}\circ \pi_r$.
\end{prop}

After defining our mirror functor, $\Phi:D^b(\Et)\to \cat{FK}^0(\Ets)$, on coherent sheaves $A = \mL(\vphi)\otimes F(V,\exp(N))$ as in theorem \ref{thm:atiyah}, we will be able to extend to all indecomposable bundles by proving
$$\label{eq:isog} \Phi(\pi_{r*} A) = p_{r*}\Phi(A)$$

For morphisms, we will need to know how to send 
$$\Hom(\pi_{r_1*}\cE_1,\pi_{r_2*}\cE_2))\to \Hom(p_{r_1*}\Phi(\cE_1),p_{r_2*}\Phi(\cE_2))$$
Consider the categorical pull back of $\pi_{r_1}$ and $\pi_{r_2}$, $E_{r_1\tau}\times_\Et E_{r_2\tau}: = \{(z_1,z_2)\stb \pi_{r_1}(z_1) = \pi_{r_2}(z_2)\}$. This is a disjoint union of elliptic curves, in particular, $E_{r_1\tau}\times_\Et E_{r_2\tau} = E_{r\tau}\times \bZ_d$ for $d = \gcd(r_1,r_2)$ and $r = \mathrm{lcm}(r_1,r_2)$.\\
Let $\tilde{\pi}_{r_i,\nu}:E_{r\tau}\times \{\nu\}\to E_{r_i\tau}$  denote the restriction of the $i\ith$ term projection.  In other words, $\tilde{\pi}_{r_1,\nu}$  (resp. $\tilde{\pi}_{r_2,\nu}$) is given by composition of the translation $t_{\nu\tau}\times\bid$ ($\bid\times t_{\nu\tau}$) with the covering map $\tilde p _{r_i}:E_{r\tau}\to E_{r_i\tau}$ (defined by the lattice inclusion $<1,r_i\tau>\subset <1,r\tau>$).

Using the above properties and the fact that $\tilde p_{r_2*}\circ \tilde p_{r_1}^*$ and $p_{r_1}^*\circ p_{r_2*}$ are isomorphic functors (and defining $\tilde \pi_{r_i,\nu}$ analogously), we have the following commutative diagram.

$$\label{thm:isogmorph}\xymatrix{
\Hom(\pi_{r_1*}\cE_1,\pi_{r_2*}\cE_2)
 \ar[d]_{\Phi_\tau} &\overset{\sim}{\longrightarrow} & \displaystyle\bigoplus_{\nu=1}^d \Hom(\tilde{\pi}^*_{r_1,\nu}\cE_1,\tilde{\pi}^*_{r_2,\nu}\cE_2) 
\ar[d]^{\bigoplus \Phi_\tau}\\
\Hom(\Phi(\pi_{r_1*}\cE_1),\Phi(\pi_{r_2*}\cE_2)) \ar[d]_{\wr} &&\displaystyle\bigoplus_{\nu=1}^d \Hom(\Phi(\tilde{\pi}^*_{r_1,\nu}\cE_1)),\Phi(\tilde{\pi}^*_{r_2,\nu}\cE_2)) \ar[d]^{\wr}\\
\Hom(p_{r_1*}\Phi(\cE_1),p_{r_2*}\Phi(\cE_2))&\overset{\sim}{\longrightarrow}  &\displaystyle\bigoplus_{\nu=1}^d \Hom(\tilde{p}^*_{r_1,\nu}\cE_1,\tilde{p}^*_{r_2,\nu}\cE_2)
}$$

\begin{rem}\label{rmk:diagramsig}
The significance of this diagram is that by its definition $\pi_{{r_i},\nu}(\mL(\vphi)\otimes F(V,\exp(N)))$ is also a bundle of the form $\mL(\vphi)\otimes F(V,\exp(N))$ (as in Atiyah's classification).  Thus the diagram makes it sufficient to define $\Phi$ on just bundles of this form. 
\end{rem}

\bigskip\bigskip
\subsection{The Equivalence of $\db(\Et)$ and $\cat{FK}^0(E^\tau)$}\hfill\\

To prove the following equivalence (in the transversal case), \cite{PZ} constructed the functor $\Phi_\tau$, which we will describe below.  Later, by extending this functor to the non-transversal case and using the additive category $\cat{FK}^0(E^\tau)$, \cite{Kreuss} made the following rigorous statement.

\begin{thm}[Main Theorem \cite{PZ,Kreuss}]{\label{thm:Kreuss}$\Phi_\tau:\db(\Et)\to \cat{FK}^0(E^\tau)$ is an equivalence of additive categories compatible with the shift functors.}\end{thm}

Recall that chain shifts in $\cat{FK}^0(E^\tau)$ correspond to deck transformations $\alpha\to\alpha+1$.  By Serre duality and additivity, it is sufficient to define $\Phi$ on indecomposable coherent sheaves.  Moreover, with regards to $\Phi$'s image on objects, Atiyah's classification theorem (coupled with \ref{eq:isog}) tells us we only need to define $\Phi$ on only vector bundles of the form $A = \mL(\vphi)\otimes F(V,\exp(N))$ and indecomposable torsion sheaves, $S = S(x,V,N)$.  We will be able to treat morphisms similarly using the above diagram \ref{thm:isogmorph} and lemma \ref{thm:H0}. We will not repeat the proof that composition holds, this can be found in \cite{PZ,Kreuss}.\\

\bigskip\bigskip
\subsection{Mirror Objects}\hfill\\

If  $A = \mL(\vphi)\otimes F(V,\exp(N))$ for $N\in \End(V)$ nilpotent with nullity one and $\vphi=(t_{a\tau+b}^*\varphi_0)\cdot \varphi_0^{n-1}$, then $\Phi_\tau(A) = (L,\alpha,M)$, where
\begin{itemize}
 \item $L$ is the submanifold of $E^\tau$ with lift parameterized by $t\mapsto (a+t,(n-1)a + nt)$.
 \item $\alpha$ is the unique appropriate (i.e. $e^{i\pi\alpha}) = \frac{n+im}{\sqrt{1+n^2}}$) real number such that\\$-\frac{1}{2}<\alpha < \frac{1}{2}$.
 \item $M = e^{-2\pi i b}\exp(N)$.
\end{itemize}
So in other words we take this rank $k$-bundle to a null-graded Lagrangian of slope $n$ (the degree of $L(\vphi)$) and $y$-intercept at $-a$  ($x$-intercept at $a/n$) with a rank $k$ local system of monodromy $e^{-2\pi i b}\exp(N)$.\\

Notice how, as in the simple case, (B-side) translations in the real direction correspond to (A-side) shifts of the Lagrangian submanifold and (B-side) translations in the $\tau$-direction (A-side) shift  of the monodromy.\\

We extend $\Phi$ to all vector bundles by defining $\Phi(\pi_{r*}A) = p_{r*}\Phi_{r\tau(A)}$.  Notice that as $\Phi_r\tau(A)$ has an integer slope, $\rest{p_r}_L$ defines an isomorphism $L\arriso p_r(L)$.  In particular the only affect of $p_{r*}$ on $\Phi(A)$ (for $A$ the particular case above) is sending Lagrangian $L$ (given by line $y=nx+a/n$) to $p_r(L)$ (given by line $y=nx/r + ra/n$).\\

Now for the final case of an indecomposable torsion sheaf.  Let $S = S(-a\tau - b,V,N)$, then $\Phi_\tau(S) = (L,1/2,e^{-2\pi i b}\exp(N))$.  Where $L$ is the given by the vertical line with $x$-intercept $a$.\\

\bigskip\bigskip
\subsection{Morphisms of Vector Bundles}\hfill\\

For morphisms, we can again use additivity and compatibility with the shift functors to reduce to the case of defining 
$$\Phi_\tau:\Hom_{\db(\Et)}(A_1,A_2[n])\to\Hom_{\cat{FK}^0(\Ets)}(\Phi_\tau(A_1),\Phi_\tau(A_2)[n])$$ for any indecomposable coherent sheaves $A_1$ and $A_2$.  By our construction of the Fukaya category on the A-side, and the fact that we are on a curve on the B-side, both sides of this map vanish for $n\notin \{0,1\}$.  Using the isomorphisms then of Serre duality (lemmas \ref{lem:Bserre} and \ref{thm:sympserre}), it is sufficient to only consider the case when $n=0$.\\

Let $A_i = \mL(\vphi_i)\otimes F(V_i,\exp(N_i))$ be as in Atiyah's classification.  Using the isomorphism from lemma \ref{thm:H0}, we will define $\Phi$ as a map $\Hom(A_1,A_2)\cong H^0(\mL(\vphi_1\inv\vphi_2))\otimes \Hom(V_1,V_2)\to \Hom_{\cat{FK}^0(\Ets)}(\Phi(A_1),\Phi(A_2))$.\\
Using a basis $\theta\otimes f$ for this B-side space, we define $\Phi$ as follows (using the notation for section \ref{sec:simple}).  Let $\theta = t^*_{\delta_{21}\tau + \beta_{21}} f^{(n_2-n_1)}_k$, for $\delta_{21} := \frac{n_2x_2-n_1x_1}{n_2-n_1}$ and $\beta_{21}:= \frac{\beta_2 - \beta_1}{n_2-n_1}$.  We define 
$$\Phi(\Psi(\theta\otimes f)) = e^{\pi i \delta_{21}^2(n_1-n_2)}\exp[\delta_{21}(N_2-N_1^*-2\pi i (n_2-n_1)\beta_{21})]\cdot f\cdot e_k$$
for $e_k$ the corresponding intersection point, which is given by
$$e_k = \left(\frac{n_1x_1 - n_2x_2 + k}{n_1-n_2},\frac{n_1}{n_1-n_2}[n_2(x_1-x_2)+k]\right) \:\text{ for } k\in \bZ/(n_1-n_2)\bZ$$
By remark \ref{rmk:diagramsig}, this defines $\Phi$ on all vector bundle morphisms.\\

\bigskip\bigskip
\subsection{Morphisms of Torsion Sheaves}\hfill\\

First the mixed cases.  Recall that we can write any indecomposable torsion sheaf on $\Et$ in the form $S = S(-a\tau - b,V',N')$ as defined in section \ref{sec:EllDb}.  Let $S$ be such a torsion sheaf and $A = \mL(\vphi)\otimes F(V,\exp(N))$ a locally free sheaf as above with $\vphi = t_{\alpha\tau + \beta}^*\vphi_0\cdot (\vphi_0)^{n-1}$.  $\Hom(A,S) = 0$.  On the symplectic side we know the same is true as $\Phi(S)$ has $\alpha=1/2 >\alpha_A$ for $\alpha_A$ the grading of $\Phi(A)$.  $\Hom(A,S) = \Hom(V, V') = V^*\otimes V$.  The Lagrangians corresponding to $S$ and $A$ only intersect at a single point.  Thus we can define $\Phi:V^*\otimes V'\to V^*\otimes V'$ as the vector space operator
$$\Phi_\tau = e^{-\pi i\tau(na^2-2a\alpha)+ 2\pi i(a\beta+b\alpha-nab)}\cdot \exp[(na-\alpha)\bid_{V^*}\otimes N' + \alpha N^*\otimes\bid_V]$$

We extend to these morphisms to morphisms to arbitrary vector bundles by the following commutative diagram:
$$\xymatrix{
\Hom(A,\pi_r^*S) \ar[d]^{\Phi_{r\tau}}&\arriso& \Hom(\pi_{r*}A, S)\ar[d]^{\Phi_{\tau}}\\
\Hom(\Phi(A),\Phi(\pi_r^*S))\ar[d]^{\wr} &\arriso& \Hom(\Phi(\pi_{r*}A), \Phi(S))\ar[d]^{\wr}\\
\Hom(\Phi(A),p_r^*\Phi(S)) &\arriso& \Hom(p_{r*}\Phi(A), \Phi(S))
}$$

Now all that remains is to define morphisms between two indecomposable torsion sheaves.  Let $S_i = S(-a_i\tau - b_i,V_i,N_i)$.  Recall these sheaves each have supports at a single point.  To have nonzero morphisms, we need that point to be shared.  On the symplectic side we have two vertical lines with $x$-intercepts $a_i$ and monodromy operators $M_i = e^{-2\pi i b_i}\exp(N_i)$.  Thus we only have the trivial morphism unless $a_1=a_2$ in which case these lines are the same and\\
$$\Hom(\Phi(S_1),\Phi(S_2))=\{f\in \Hom(V_1,V_2)\stb f\circ M_1 = M_2\circ f\}$$

Also note that if $b_1\neq b_2$,  $M_1$ and $M_2$ do not share any eigenvalues (as $N$ is nilpotent it has no nonzero eigenvalues) and thus can not be conjugates of each other in the space of endomorphisms.\\

If $a_1\tau + b_1 = a_2\tau + b_2$, then 
\begin{align*}\Hom(S_1,S_2) &= \{f\in \Hom(V_1,V_2)\stb f\circ N_1= N_2\circ f\}\\
&= \{f\in \Hom(V_1,V_2)\stb f\circ M_1 = M_2\circ f\}\\
&= \Hom(\Phi(S_1),\Phi(S_2))\end{align*}

This completes our construction of $\Phi$.\\

%% file: HomAlg.tex
\section{Homological Algebra Background}
In this appendix we will give a minimalist review of the homological algebra used in this thesis.  More thorough reviews of these subjects are given in \cite{Keller, ThDerived}.\\

\bigskip\bigskip
\subsection{Chain Complexes}\label{sec:chaincomp}\hfill\\

Before we start our discussion of derived categories let us  review some basic terminology used in homology theory.\\
Let $\cA$ be an abelian category.\footnote{A category is called \emph{abelian} if all its homsets are abelian groups and composition of morphisms gives a bilinear operator.  For example, the category of coherent sheaves on a manifold is abelian.}
\begin{defn}
A \emph{chain complex} in $\cA$, $(A_\bullet,d_\bullet)$, is a sequence of objects connected by morphisms (called boundary operators) $d_n:A_n\to A_{n-1}$ such that $d_n\circ d_{n+1} = 0$ for all $n$.\end{defn}

Associated to any chain complex is a homology, $H_n(A_\bullet) = \ker{d_{n-1}}/\im{d_n}$.\\

\begin{defn}
A \emph{chain map} between to chain complexes $(A_\bullet,d^A_\bullet)$ and $(B_\bullet,d^B_\bullet)$ is a sequence $f_\bullet$ of homomorphisms $f_n:A_n\to B_n$ such that $d^B_n\circ f_n = f_{n-1}\circ d^A_n$ for all $n$.\end{defn}

Chain maps induce maps between homologies. If the induced maps are isomorphisms, then the chain map is called a \emph{quasi-isomorphism}.  

\ex  Continuous maps induce chain maps on singular and smooth\footnote{Actually we only need a continuous map between topological spaces admitting a differential structure as de Rham cohomology cannot tell the difference between homeomorphic spaces.} maps induce cochain maps on de Rham cochain complexes.\\

\begin{defn}
For semantic convenience let us  also recall that a morphism between two graded structures $A_\bullet$ and $B_\bullet$ is said to be \emph{homogeneous\footnote{We will in general omit this word ``homogenous" when speaking of such maps as its implication is clear from the context.} of degree $k$} if it maps $A^{n}\to B^{n+k}$ for all $n$.\end{defn}

\ex Chain maps are degree 0 maps that preserve the boundary operator, which itself is defined as a degree -1 map that squares to zero.  Cochain maps are degree 0 maps that preserve the coboundary operator, which itself is defined as a degree 1 map that squares to zero.

\begin{defn}
We say two chain maps $f_\bullet:A_\bullet\to B_\bullet$ and $g_\bullet:A_\bullet\to B_\bullet$ are \emph{homotopic}\footnote{This name comes the fact that homotopic maps of topological spaces induce these maps in the case of singular chains.} if there exists some degree -1 map $h_\bullet:A_\bullet\to B_\bullet$ such that $f_\bullet-g_\bullet = d_B h_\bullet + h_\bullet d_A$ (i.e. $f^n-g^n = d_B^{n-1} h^n + h^{n+1}d_A^n$ for all $n$).\end{defn}

\begin{rem}
If two chain maps are homotopic, then they induce the same maps between homologies.\end{rem}

\bigskip\bigskip
\subsection{The Derived Category}\label{sec:DerivedCat}\hfill\\

\begin{defn}
Let $\cA$ be some abelian category and consider, $\cC\cA$ the \emph{category of chain complexes} in $\cA$, where objects are chain complexes (of objects in $\cA$) and morphism are chain maps.  The \emph{homotopy category} of $\cA$, denoted $\cat{HA}$, has the same objects as $\cA$, but we consider the morphisms up to homotopy equivalence as defined above.\end{defn}

Before defining the derived category of $\cA$, we need to discuss quasi-isomorphisms and the rough concept of ``localization".  Localization is, roughly speaking, the process of formally adding inverse elements to a subset of morphisms that we are ``localizing over".  For example, in algebraic geometry, we may, wishing to better understand the local geometry of a point (or subvariety) in some variety, localize our coordinate ring over complement, $S$, of the maximum (prime) ideal (of the coordinate ring, $R$) corresponding to that point (subvariety). For this example $S$ is closed under multiplication, which allows us to define this localization as the ring of fractions over $S$, $S\inv R: = \{\frac{r}{s}\stb s\in S \an r\in R\} = R\times S/~$, where $(r_1,s_1)~(r_2,s_2)\iff \exists t\in S$ such that $t(r_1s_2 - r_2s_1)$ ($t$ is necessary for transitivity).\\  

\begin{defn}
In general we define a quasi-isomorphism to be a morphism that descends to any isomorphism of homologies.  To be more specific (so we can define the derived category below), notice that the homology functor $H^n:\cat{CA}\to \cA$, which sends a chain (chain map) to its $n\ith$ object (morphism), descends naturally to a functor $\cat{HA}\to \cA$.  We will say a morphism\footnote{From now on we will drop the bullets when talking about complexes and their morphisms.} $s\in \Hom_\cat{HA}(A,B)$ is a \emph{quasi-isomorphism} if the induced morphism $H^n s:A\to B$ is invertible for all $n$.\end{defn}

Below we will define the derived category of $\cA$ to be the localization of $\cat{HA}$ over the space of quasi-isomorphisms.  First we a few observations about quasi-isomorphisms.

\begin{lem}\hfill\\
\begin{enumerate}
\item[a)]  Isomorphisms and compositions of quasi-isomorphisms are quasi-isomorphism.\\
\item[b)]  Each diagram $A'\xleftarrow{s} A\xrightarrow{f} B$ (or $A'\xrightarrow{f'} B'\xleftarrow{s'} B$) of $\cat{HA}$, for $s$ (resp. $s'$) a quasi-isomorphism, can be embedded into a commutative square:\\
$${
\xymatrix{
A \ar[d]_s \ar[r]^f &B\ar[d]^{s'}\\
A' \ar[r]_{f'} &B'}
}$$

\item[c)]  Let $f\in \text{Mor}(\cat{HA})$.  Then there exists a quasi-isomorphism $s$ such that $sf=0$ if and only if there exists a quasi-isomorphism $t$ such that $ft=0$.
\end{enumerate}\end{lem}
\begin{proof}A proof of this can be found in section 1.6 of \cite{kashiwara2002sheaves}.\end{proof}

\begin{rem}{Notice (a) is true in general, but (b) and (c) hold due to homotopy equivalence.}\end{rem}

Now on to defining the title of this section.\\
\begin{defn}
The \emph{derived category} of $\cA$, denoted $\cDA$, has the same objects $\cat{HA}$, but the morphisms, $Hom_\cDA(A,B)$, are given by pairs\footnote{These pairs are sometimes called \emph{roofs} or more appropriately, \emph{left fractions}.} $A\xrightarrow{f} B'\xleftarrow{s} B$, denoted $(f,s)$, of a morphism and quasi-isomorphism under the following equivalence.  We say $(f,s)\sim(f',s')$ if and only if there is a commutative diagram in $\cat{HA}$:\\
$$\xymatrix{
&B'\ar[d]&\\
A\ar[ur]^f\ar[r]^{f''}\ar[dr]_{f'}&B'''&B\ar[ul]_s\ar[l]_{s''}\ar[dl]^{s'}\\
&B''\ar[u]&
}$$
\end{defn}

We define composition of morphisms in $\cDA$ by $(f,s)\circ(g,t) = (g'f,s't)$, where $g'$ and quasi-isomorphism, $s'$ are defined by the following commutative diagram in $\cat{HA}$:\\
$${
\xymatrix{
&& C''&&\\
&B'\ar[ur]^{g'}&& C'\ar[ul]_{s'}&\\
A \ar[ur]^{f} &&B\ar[ul]_s\ar[ur]^g && C\ar[ul]_t}
}$$\\


%

\begin{lem}
Let $\cA$ be an abelian category.  Then, for any $X,Y\in \cA$ and $k\in\bZ$, there is a canonical isomorphism\\
$$\ext^k_\cA(X,Y) \cong \Hom_{\db\cA}(X,Y[k])$$
\end{lem}\begin{proof}This is true essentially by definition and is often taken as such.  An explanation can be found in \cite{Keller} or in section I.6 of \cite{HartResDua}.\end{proof}

%


\begin{rem}The derived category has a triangulated structure, which is often used in mirror symmetry.  We will not talk about this here.  For a review of the derived category of an abelian category and its triangulated structure, see \cite{Keller, ThDerived}.\end{rem}

%% file: SpecialLags.tex
\section{Minimal, Calibrated, and Special Lagrangian Submanifolds}\label{sec:SL}

Below we will define special Lagrangians and a grading which can be placed on them.  First we will give some background by discussing minimal and calibrated submanifolds.\\

\bigskip\bigskip
\subsection{Minimal Submanifolds}\hfill\\

Roughly speaking, a minimal submanifold is a submanifold with the property that small deformations of its embedding do not affect its volume. We will be more specific below, but first let us review some basic notions from Riemannian Geometry.  For a more in depth discussion on these topics, see \cite{Xin,Hitch99,Hitch97}.\\
Let $S$ and $M$ be $n$ and $(n+k)$-dimensional smooth manifolds respectively.  An isometric immersion, $S\imm M$ gives us a smooth splitting of the tangent bundle of $M$ into parts tangent and normal to $S$, $\evat{TM}_S = TS\oplus NS$.  The \emph{second fundamental form} of $S\imm M$ is a symmetric bilinear $NS$-valued form on $S$ defined in terms of the Levi-Civita connections on $M$ and $S$ as $B(X,Y) := \nabla^M_XY - \nabla^S_XY$.\\
The \emph{mean curvature} of $S\imm M$ is then given by the average of the eigenvalues of the second fundamental form, $H = \frac{1}{n}\text{Trace}(B)$.\\
$S$ is said to be a \emph{totally geodesic submanifold} in $M$ if all geodesics in $S$ are geodesics in $M$.
This is true if and only if $B\equiv0$.\footnote{The less-trivial direction can be seen by observing that $B$ is symmetric and thus $B(X,Y)$ at a point in $S$ depends only on the values of $X$ and $Y$ at that point.}\\
$S$ is said to be a \emph{minimal submanifold} in $M$ if $H\equiv0$.\\

To see why such submanifolds are minimal in the sense described at the beginning of this section, recall the \emph{first variation formula} of Riemannian Geometry:\\\\
\begin{thm}Let $S$ be a compact Riemannian manifold and $f:S\imm M$ some isometric immersion with mean curvature $H$.  Let $f_t$ for $|t|<\epsilon$, $f_0=f$, be a smooth family of immersions satisfying $\evat{f_t}_{\partial S} = \evat{f}_{\partial S}$.  Let $V = \evat{\frac{\partial f_t}{\partial t}}_{t=0}$.  Then\\
$$\evat{\ddt \text{vol}(f_t (S))}_{t=0} = -\int_S <nH,\evat{\frac{\partial f_t}{\partial t}}_{t=0}>d\,\text{vol}.$$
\end{thm}\begin{proof}
This is straight-forward application of Stoke's formula and the fact that\\
$\evat{\ddt \det(A(t))}_{t=0} = \text{Trace}(A'(0))$.  See \cite{Xin} for a detailed proof.\end{proof}

\begin{rem} The R.H.S. of this equation gives the Euler-Lagrange equation $H=0$, which explains the above definition of a minimal submanifold.\end{rem}

\ex It's immediate (from the above theorem or definition) that one-dimensional minimal submanifolds are geodesics.\\


\bigskip\bigskip
\subsection{Calibrated Submanifolds}\hfill\\

\begin{defn}
Let $M$ be some Riemannian manifold and $\vol_V$ denote the induced volume form on any given subspace $V\subset T_xM$.  A $k$-form, $\eta$, is said to be a \emph{calibration} on $M$ if it is closed and, for all oriented $k$-dimensional subspaces $V\subset T_xM$, $\evat{\eta}_V = \lambda \text{vol}_V$ for some $\lambda\leq 1$.
We say a submanifold $N\hookrightarrow M$ is \emph{calibrated} with respect to calibration $\eta$ (or \emph{$\eta$-calibrated}) if for all $x\in N$, $\evat{\eta}_{T_xN} = \text{vol}_V$.\\
\end{defn}

\emph{Proposition:  }Let $(M,g)$ be a Riemannnian manifold, $\eta$ a calibration on $M$, and $N$ a compact $\eta$-calibrated submanifold of $M$.  Then $N$ is volume-minimizing in its homology class.\\
\begin{proof}
Let $[N] = [N']\in H_{dim(N)}(M)$
$$\int_N \mathrm{vol}_N = \int_N \eta = \int_{N'} \eta \leq \int_{N'} \mathrm{vol}_{N'}$$
\end{proof}

\rem{This tells us that the volume of calibrated submanifolds is unchanged by small variations, in other words, calibrated submanifolds are minimal submanifolds.}

\ex Any submanifold calibrated by a 1-form is a geodesic.\\

\bigskip\bigskip
\subsection{Special Lagrangian Submanifolds}\hfill\\

In order to define special Lagrangian submanifolds, we need to make the following observation.\\
Let $M$ be a (complex) $n$-dimensional Calabi-Yau manifold with Calabi-Yau form $\Omega$ and \kah{} form $\omega$.  The volume form induced by the \kah{} metric is then given by $\omega^n/n! = \frac{1}{n!}\sum dz$.\\
Recall that a Lagrangian submanifold is a (real) $n$-dimensional submanifold $L\hookrightarrow M$ such that $\rest{\omega}_L=0$.
Calabi-Yau manifolds are by definition compact and thus all holomorphic volume forms on them are locally constant.  We can certainly always scale $\Omega$ (uniquely up to phase) such that we have $\omega^n/n! = (-1)^{n(n-1)/2}(i/2)^n\Omega\wedge\bar\Omega$.  This scaled $\Omega$ is then the unique Calabi-Yau form which is a nontrivial calibration on $M$.  For any Lagrangian submanifold, $L\hookrightarrow M$,  $Vol_L = e^{i\pi\theta}\Omega_L$ for some $\theta:L\to \bR$.\footnote{This is easy to see in local holomorphic/Darboux coordinates.  $\rest{\Omega_L}\propto \prod_k (dx_k+idy_k) \neq 0$ if and only if $L$ is a Lagrangian submanifold.}  We will call this $\theta$, the phase of $L$ (see discussion in section \ref{sec:phaseL} below).
\begin{defn}
The \emph{special Lagrangian} submanifolds of $M$ are the Lagrangian submanifolds of $M$ such that $\theta$ is constant.\footnote{Sometimes special Lagrangian submanifolds are defined using a fixed phase.}
\end{defn}
Equivalently we could define them to be the submanifolds of $M$ calibrated by $\reo$ under some choice of phase (i.e. the phase of $L$).\\

\ex By the above discussions of calibrated and minimal submanifolds it is immediate that all special Lagrangian submanifolds of a (complex) 1-dimensional Calabi-Yau manifold (i.e. an elliptic curve) are geodesics.  It is not hard to see that the converse is also true as the geodesics are exactly the submanifolds of constant phase.  Some less trivial examples and a proof of the following theorem can be found in \cite{Xin}.\\

\begin{thm}A submanifold $L$ of $\bC^n = \bR^{2n}$ is both Lagrangian and minimal if and only if $L$ is  special Lagrangian.\end{thm}

\bigskip\bigskip
\subsection{Graded Lagrangian Submanifolds}\label{sec:phaseL}\hfill\\

$\theta$ is obviously unique only up to an integer shift (and well-defined, mod $\bZ$, only with respect to a fixed $\Omega$).  Given a fixed $\Omega$, in the case where $\theta$ is constant (i.e. $L$ is special), we call this choice the \emph{grading} of $L$.  A \emph{graded special Lagrangian submanifold} is then a corresponding pair $(L,\theta)$.  This is equivalent to the concept of grading used below in the example of homological mirror symmetry for elliptic curves.  The shift $(L,\theta)[n]:= (L\theta+n)$ mirrors the $n\ith$ shift functor on complexes of coherent sheaves.  Notice that the shift $(L,\theta)\mapsto (L\theta+1)$ is equivalent to reversing the orientation of $L$ induced by $\vol_L$ ($\vol_L\mapsto-\vol_L$ sends $e^{i\pi\theta}\to e^{-i\theta}$).\\

We can generalize this concept of grading to other Lagrangian submanifolds by defining the \emph{average phase} of $L$, $\phi(L)$, given modulo $\bZ$ by
$$\frac{\int_L\Omega}{\int_L \vol_L} = e^{i\pi\phi(L)}$$
This is clearly invariant under Hamiltonian deformations of $L$ (as this ratio is holomorphic and thus must be constant) and, assuming $L$ is homologous to a special Lagrangian submanifold, matches the restricted concept of grading given above.\\
